\documentclass[leqno,11pt]{article}

\usepackage{amsmath,amsthm,amscd,amssymb,amsbsy,amstext,mathtools,yhmath,cite}
\usepackage{pdfpages}
\usepackage{enumerate}
\usepackage{mathrsfs}
\usepackage[utf8]{inputenc}

\usepackage{amsfonts}
\usepackage[T1]{fontenc}

\usepackage{authblk}
\usepackage{url}
\usepackage{hyperref}

\usepackage{appendix}

\newtheorem{theorem}{Theorem}[section]
\newtheorem{lemma}{Lemma}[section]
\newtheorem{proposition}{Proposition}[section]

\newtheorem{problem}{Problem}

\theoremstyle{remark}
\newtheorem{remark}{Remark}[section]

\theoremstyle{definition}
\newtheorem{definition}{Definition}[section]

\numberwithin{equation}{section}

\newcommand{\norm}[1]{\| {#1}\| }
\newcommand{\set}[1]{\left\{#1\right\}}
\newcommand{\void}{\varnothing}
\newcommand{\abs}[1]{\left|#1\right|}
\newcommand{\dist}[2]{\mathrm{dist}({#1}, {#2})}
\newcommand{\brac}[1]{\left(#1\right)}

\DeclareMathOperator{\diam}{diam}
\DeclareMathOperator{\supp}{supp}

\newcommand{\Rn}{\mathbb{R}^n}
\newcommand{\R}{\mathbb{R}}

\renewcommand{\d}{\partial}
\newcommand{\da}{\d^\alpha}
\newcommand{\grad}{\nabla}

\newcommand{\T}{\mathcal{T}}

\newcommand{\jet}{\mathcal{J}}
\renewcommand{\P}{\mathcal{P}}

\renewcommand{\phi}{\varphi}
\renewcommand{\void}{\varnothing}
\newcommand{\pos}{[0,\infty)}

\newcommand{\Cs}{{\mathcal{C}^s}}
\newcommand{\Csh}{{\dot{\mathcal{C}}^s}}
\newcommand{\Fs}{{\mathcal{F}^s}}
\newcommand{\Fsh}{{\dot{\mathcal{F}}^s}}

\newcommand{\JC}{\jet\Cs}
\newcommand{\JCh}{\jet\Csh}
\newcommand{\JF}{\jet\Fs}
\newcommand{\JFh}{\jet\Fsh}

\newcommand{\floor}[1]{\lfloor#1\rfloor}

\title{Roots, trace, and extendability of flat nonnegative smooth functions}
\author[\dagger]{Fushuai Jiang}
\affil[\dagger]{Department of Mathematics\\
University of Maryland at College Park\\
4176 Campus Dr, College Park, MD 20742, USA\\
fsjiang@umd.edu
}
\date{}

\begin{document}

\maketitle

\begin{abstract}
    Building on the univariate techniques developed by Ray and Schmidt-Hieber, we study the class $\mathcal{F}^s(\mathbb{R}^n)$ of multivariate nonnegative smooth functions that are sufficiently flat near their zeroes, which guarantees that $F^r$ has H\"older differentiability $rs$ whenever $F \in \mathcal{F}^s$.
    We then construct a continuous Whitney extension map that recovers an $\mathcal{F}^s$ function from prescribed jets. Finally, we prove a Brudnyi-Shvartsman  Finiteness Principle for the class $\mathcal{F}^s$, thereby providing a necessary and sufficient condition for a nonnegative function defined on an arbitrary subset of $\mathbb{R}^n$ to be $\mathcal{F}^s$-extendable to all of $\mathbb{R}^n$. 
\end{abstract}

\section{Introduction}

For $ s \in \R $, we use $ \floor{s} $ to denote the largest integer strictly smaller than $ s $.
For a convex domain $\Omega \subset \Rn$ and a real number $s > 0$, let $\Cs(\Omega)$ denote the H\"older-Zygmund space of $ \floor{s} $-times continuously differentiable functions whose derivatives up to order $ \floor{s} $ are bounded, and the $ \floor{s} $-th order derivatives are H\"older continuous with H\"older exponent $ s - \floor{s} $. The space becomes a Banach space when equipped with the norm
\begin{equation*}
	\norm{F}_{\Cs(\Omega)}:= 
	\max_{0 \leq m <s}\sup_{x\in \Omega}\abs{\grad^m F(x)} +  
	\sup_{x \neq y,\, x, y \in \Omega} \frac{\abs{
			\grad^{\floor{s}}F(x) - \grad^{\floor{s}}F(y)
	}}{\abs{x-y}^{s-\floor{s}}}.
\end{equation*}

Building on the univariate techniques developed by K. Ray and J. Schmidt-Hieber \cite{RSH17}, we study the smoothness property of the roots of multivariate nonnegative H\"older-differentiable functions. Furthermore, we provide a description of their trace to subsets of $\Rn$ and provide criteria for their extendability from incomplete data.

The question of differentiability of square roots was first studied by G. Glaeser \cite{G63},
 who showed the sharp result that any nonnegative univariate function that is $2$-flat ($f(x) = 0 \Longrightarrow f''(x) = 0$) admits a continuously differentiable square root. Additional flatness conditions then become necessary for a smooth function to have a square root with additional regularity. All the roots considered in this paper are nonnegative, in contrast with the notion ``admissible roots'' that are allowed to change signs \cite{root-Bony06}.

In this paper, we consider a particular positive cone $\Fs(\Omega) \subset \Cs(\Omega)$ equipped with a seminorm
\begin{equation*}
	\norm{F}_{\Fsh(\Omega)}:= 
	\begin{cases}
		\max\limits_{1 \leq m < s}\sup\limits_{x \in \Omega}\brac{ \frac{\abs{\grad^{m}F(x)}^s}{F(x)^{s-m}} }^{1/m} &\text{ if } s \geq 1\\
		0 &\text{ if } 0 \leq s < 1
	\end{cases},
\end{equation*}
such that $F \in \Fs$ is a sufficient condition for $F^r$ to have H\"older regularity $rs$. We define
\begin{equation*}
    \norm{F}_{\Fs(\Omega)}:= \norm{F}_{\Cs(\Omega)}+\norm{F}_{\Fsh(\Omega)}.
\end{equation*}

\begin{theorem}\label{thm:root}
\newcommand{\Crs}{\mathcal{C}^{rs}}
\newcommand{\Frs}{\mathcal{F}^{rs}}
For $r \in (0,1]$, $s > 0$, a convex domain $\Omega\subset \Rn$, and all nonnegative $F \in \Fs(\Omega)$,
\begin{equation*}
    \norm{F^r}_{\Crs(\Omega)} \leq \norm{F^r}_{\Frs(\Omega)} \leq C(n,r,s)\norm{F}_{\Fs(\Omega)}^r.
\end{equation*}
\end{theorem}

Note that the univariate version ($n=1$) of Theorem \ref{thm:root} was proved by Ray and Schmidt-Hieber\cite{RSH17}. The main difference between the proof of Theorem \ref{thm:root} and its univariate counterpart is in the analysis of the combinatorial form of the multivariate Fa\'a di Bruno's formula (Theorem \ref{thm:FdB}). We supply the proof of Theorem \ref{thm:root} in Appendix \ref{section:thm:root} for completeness.

In view of Theorem \ref{thm:root}, $F \in \Fs(\Omega)$ is a sufficient condition for $F^{1/2} \in \mathcal{C}^{s/2}(\Omega)$, thereby overcoming the limitation described in \cite{root-Bony06,G63}. The quantity $\norm{F}_{\Fsh(\Omega)}$ measures the flatness of $F$ near its zeros in $\Omega$. In particular, if $F \in \Fs(\Omega)$ and $F(x_0) = 0$ for some $x_0 \in \Omega$, then $\grad^m F(x_0) = 0$ for all $0 \leq m < s$. The class $\Fs(\Omega)$ contains, for example, constant functions, nonnegative $\Cs$ functions uniformly bounded away from zero on $\Omega$, and functions of the form $\abs{x-x_0}^sG(x)$ with $G$ nonnegative and uniformly bounded away from zero. See Section \ref{sect:Fs basic} for other basic properties of $\Fs$, and
\cite{RSH17} for further comparison between $\Fs$ and other criteria considered in the literature. 

The convexity assumption on $\Omega$ in Theorem \ref{thm:root} can be relaxed, as long as a suitable variant of the mean value theorem holds. We will not pursue such generality in this paper.

Next, we consider two Whitney-type extension problems on how to characterize the restriction of $\Fs$ functions to a closed set (with no assumption on differentiability structures), and how to recover a function $F \in \Fs(\Rn)$ from partial data with $ \norm{F}_{\Fs(\Rn)} $ as small as possible, up to a constant factor depending only on the dimension and the smoothness. We call such an $ F $ quasi-optimal.

First, we consider the classical jet-extension problem, i.e., reconstruction of an $\Fs$ function from given Taylor expansions. 

\begin{problem}\label{prob:jet}
    Let $E\subset \Rn$ be a closed set. For each $x \in E$, let $P_x$ be a polynomial of degree at most $\floor{s}:= \max\set{m \in \mathbb{N}_0 : m < s}$ (a jet of order $ \floor{s} $) with $P_x(x) \geq 0$. 
    \begin{enumerate}[(A)]
        \item Does there exist an $F \in \Fs(\Rn)$ such that $\jet_x F \equiv P_x$ for all $x \in E$, i.e., the $\floor{s}$-jet of $F$ at $x$ agrees with $P_x$ for all $x \in E$?
        \item If so, how do we find such an $F$ taking all the prescribed jets, such that $\norm{F}_{\Fs(\Rn)}$ is as small as possible, up to a multiplicative factor $C(n,s)$?
    \end{enumerate}
\end{problem}

Let $\jet_E F$ denote the parameterized family of jets $(\jet_x F)_{x\in E}$. Taylor's theorem, taking flatness into consideration, states that $\norm{\jet_E F}_{\JF(E)}\leq C(n,s)\norm{F}_{\Fs(\Rn)}$, where $\norm{\,\cdot\,}_{\JF(E)}$ is a norm on the positive cone $\JF(E)$ of flat jets associated with the class $\Fs$ (see \eqref{eq:Whitney norm} below). Our solution to Problem \ref{prob:jet} is captured by the following theorem.

\begin{theorem}\label{thm:whitney-inverse}
   	Let $ E\subset\Rn $ be a closed set and let $ (P_x)_{x\in E} $ be a field of jets of order $ \floor{s} $ with $ P_x(x) \geq 0 $ for all $ x \in E $. Then $ (P_x)_{x\in E} $ extends to an $ \Fs(\Rn) $ function
      if and only if $\norm{(P_x)_{x\in E}}_{\JF(E)} < \infty$. Moreover, 
    we can construct a continuous map $\T_E : \JF(E) \to \Fs(\Rn)$ such that $\jet_E\circ \T_E = Id$ on $\JF(E)$.
\end{theorem}

See Theorem \ref{thm:whitney-finite} (when $E$ is finite) and Theorem \ref{thm:whitney-closed} (when $E$ is possibly infinite) for a more detailed statement. We shall prove Theorem \ref{thm:whitney-finite} in Section \ref{sect:whitney}, and we will explain the necessary modification for proving Theorem \ref{thm:whitney-inverse} when $E$ is possibly infinite.

Next, we consider the much harder function-extension problem.

\begin{problem}[Whitney Extension Problem for $\Fs$]\label{prob:function}
Let $E\subset \Rn$ be a closed set and let $f : E \to \pos$ be a continuous function. 
\begin{enumerate}[(A)]
    \item Does there exist an $F \in \Fs(\Rn)$ such that $F = f$ on $E$?
    \item If so,
    how do we find such an $F \in \Cs(\Rn)$ such that $F = f$ on $E$ and $\norm{F}_{\Fs(\Rn)}$ is as small as possible, up to a multiplicative factor $C(n,s)$?
\end{enumerate}

\end{problem}

In this paper, we provide an answer to Problem \ref{prob:function}(A) in terms of the following ``Finiteness Principle''.

\begin{theorem}[Finiteness Principle for $\Fs(\Rn)$]\label{thm:fp-arbitrary}
    Given $ n \geq 1 $ and $ s > 0 $, there exist constants $k^\sharp = k^\sharp(n,s)$ and $C^\sharp = C^\sharp(n,s)$ such that the following holds. Given an arbitrary $E\subset \Rn$ (not necessarily finite or closed) and a function $f : E \to \pos$, $f$ extends to a $\Fs(\Rn)$ function if and only if 
    \begin{equation*}
        \sup_{S\subset E,\, \#S \leq k^\sharp }\norm{f}_{\Fs(S)}
         < \infty,
    \end{equation*}
    where $\norm{f}_{\Fs(S)}:= \inf\set{
        \norm{F}_{\Fs(\Rn)} : F \in \Fs(\Rn) 
        \text{ and } F = f \text{ on }S
        }$.
    Moreover, we have the following equivalence of norms
    \begin{equation}
        \sup_{S\subset E,\, \#S \leq k^\sharp }\norm{f}_{\Fs(S)}\leq \norm{f}_{\Fs(E)} \leq C^\sharp \cdot\sup_{S\subset E,\, \#S \leq k^\sharp }\norm{f}_{\Fs(S)}.
        \label{eq:fp-closed-quant}
    \end{equation}
\end{theorem}

To wit, the only obstruction to the existence of a global extension is the lack of control in some local extensions.

Problem \ref{prob:function} for $\Cs(\Rn)$, its variants, and the related Finiteness Principles have been extensively studied by Y. Brudnyi and P. Shvartsman \cite{BS94-W,BS01}, C. Fefferman \cite{F05-J,F05-Sh,F09-Data-3}, and
C. Fefferman, A. Israel, and G.K. Luli \cite{FIL16,FIL16+}. In Proposition \ref{prop:c2} below, we see that for $s \in (0,2]$, $\Fs(\Rn)$ consists of exactly the nonnegative functions in $\Cs(\Rn)$. Thus, Problem \ref{prob:function} for $s = 2$ has been solved in the author's joint papers \cite{JL20-Ext,JL20,JL20-Alg,FJL23}. In particular, there exist efficient algorithms to compute a quasi-optimal extension when $E$ is finite. See also \cite{JLLL23} for the one-dimensional implementation of the extension algorithm in the software package R. Note that if $0 < s \leq 1$, $\Fs$ agrees with the cone of nonnegative $\Cs$ functions. By the classical Whitney extension theorem (see Section \ref{sect:whitney}, or the Kirszbraun theorem), 
$f$ extends to $\Fs$ if and only if $f$ is nonnegative and 
Lipschitz on $(E,d_s)$, where $d_s(x,y) = \abs{x-y}^s$.

We will prove Theorem \ref{thm:fp-arbitrary} in Section \ref{sect:fp}.
The number $k^\sharp$ in Theorem \ref{thm:fp-arbitrary} resulting from our proof is unnecessarily large because it relies on the refinement procedure of an abstract object called ``shape fields'' (Definition \ref{def:sf} below) introduced in \cite{FIL16,FIL16+}. However, it can be substantially improved to be $k^\sharp = 2^{\dim\P}$, where $\P$ is the vector space of polynomials in $n$ variables of degree at most $\floor{s}$. See \cite{JLO20,BM07}.

In an upcoming paper, we will study Problem \ref{prob:function}(B). 

This is a part of the literature on extension, interpolation, and selection of functions, going back to H. Whitney's seminal works \cite{W34-1,W34-2,W34-3}. We refer the interested readers to \cite{G58,BS94-Tr, BS98,BS01,Shv08,BM07,BMP06,F05-Sh,F05-L,F05-J,F06,F07-L,F09-Data-3,F09-Int,Z98,Z99,JLO22, JLO20,FIL13,FIL16,FIL16+,FJL23,FI20-book,W34-1,W34-2,W34-3} and references therein
for the history and related problems.
We also refer the interested readers to \cite{root-Bony06,root-Bony10, RSH17} and references therein for more backgrounds on the regularity of roots. Besides their connection with the mathematical aspect of optimization and machine learning, another application of these results can be found in non-parametric statistics, where one seeks to recover a H\"older function $F$ when observing a noisy version of $F^{1/2}$ \cite{noise,RSH16, RSH17}.  \\

 \textbf{Acknowledgment.} I am grateful to my former Ph.D. advisor Kevin Luli for introducing me to Whitney's extension problems and for his valuable suggestions on this manuscript. I would like to thank Kolyan Ray and Johannes Schmidt-Hieber for their insightful comments on the multivariate counterpart of their original results. I would also like to thank the anonymous referees for their detailed suggestions of the manuscript. 
 
 Part of this material is based upon work supported by the National Science Foundation under Grant No. DMS-1439786 while I was in residence at the Institute for Computational and Experimental Research in Mathematics in Providence, RI, during the fall 2022 semester.

\section{Definitions and notations}
\label{sect:definitions}

Throughout the paper, $\Omega$ always denotes a connected open subset of $ \Rn $. Let $\floor{s}$ denote the largest integer {\em strictly} less than $s\in \R$. For a positive real $s > 0$, we use $\Cs(\Omega)$ to denote the vector space of $\floor{s}$-times continuously differentiable functions whose derivatives up to order $\floor{s}$ are bounded and $(s-\floor{s})$-H\"older continuous. We equip $\Cs(\Omega)$ with the norm
\begin{equation*}
    \norm{F}_{\Cs(\Omega)}:= 
    \max_{0 \leq m <s}\sup_{x\in \Omega}\abs{\grad^m F(x)} +  
    \sup_{x \neq y,\, x, y \in \Omega} \frac{\abs{
    \grad^{\floor{s}}F(x) - \grad^{\floor{s}}F(y)
    }}{\abs{x-y}^{s-\floor{s}}}.
\end{equation*}
Here, $\grad^m F(x)$ denotes the symmetric $m$-linear form and $\abs{\grad^m F(x)} = \brac{\sum_{\abs{\alpha}=k}\abs{\da F(x)}^{2}}^{1/2}$, where $\alpha = (\alpha_1, \cdots, \alpha_n) \in \mathbb{N}_0^n$ is a multi-index, $\abs{\alpha} := \sum_{j = 1}^n\alpha_j$, and $\da = \d_{t_1}^{\alpha_1}\cdots\d_{t_n}^{\alpha_n}$. We use the $\grad^mF$ notation when we want to emphasize on the number of derivatives taken. We will also use the homogeneous H\"older seminorm
\begin{equation*}
    \norm{F}_{\Csh(\Omega)}:= \sup_{x \neq y,\, x, y \in \Omega} \frac{\abs{
    \grad^{\floor{s}}F(x) - \grad^{\floor{s}}F(y)
    }}{\abs{x-y}^{s-\floor{s}}}.
\end{equation*}

We would like to point out the difference between $ \Cs(\Omega) $ and the non-H\"older space $ C^m(\Omega) $ consisting of $ m $-times continuously differentiable functions whose derivatives up order $ m $ are bounded and continuous, equipped with the norm $ \norm{F}_{C^m(\Omega)} := \max\limits_{0 \leq k \leq m}\sup\limits_{x \in \Omega}\abs{\grad^k F(x)} $. The closed unit ball of $ \Cs(\Omega) $ is compact in the $ C^{\floor{s}}(\Omega) $ topology whenever $ \Omega $ is bounded. We will use this fact in the proof of Theorem \ref{thm:fp-arbitrary}.
 
For a nonnegative function $F \in \Cs(\Omega)$, we define
\begin{equation*}
    \norm{F}_{\Fsh(\Omega)}:= 
    \begin{cases}
        \max\limits_{1 \leq m < s}\sup\limits_{x \in \Omega}\brac{ \frac{\abs{\grad^{m}F(x)}^s}{F(x)^{s-m}} }^{1/m} &\text{ if } s \geq 1\\
        0 &\text{ if } 0 \leq s < 1
    \end{cases},
\end{equation*}
where we adopt the conventions $\frac{0}{0}= 0$ and $\frac{a}{0} = \infty$ for $a > 0$. We use $\Fs(\Omega)$ to denote the collection of nonnegative functions $ F \in \Cs(\Omega)$ such that $\norm{F}_{\Fsh} < \infty$, and we define
\begin{equation*}
    \norm{F}_{\Fs(\Omega)}:= \norm{F}_{\Cs(\Omega)} + \norm{F}_{\Fsh(\Omega)}.
\end{equation*}
We will see in Proposition \ref{prop:Fs basic}(A) that $\norm{\cdot}_{\Fsh(\Omega)}$ defines a seminorm on the positive cone $\Fs(\Omega)\subset \Cs(\Omega)$, thus justifying the notations above. 
Since every $F \in \Cs(\Omega)$ has well-defined values on the closure $\overline{\Omega}$, we can make sense of the spaces above if we replace $\Omega$ by $K$ with $\Omega \subset K\subset \overline{\Omega}$.

We use $\P = \P^{\floor{s}}$ to denote the space of polynomials on $\Rn$ of degree no greater than $\floor{s}$. 

Let $F$ be $\floor{s}$-times continuously differentiable near $x_0 \in \Rn$, we use $\jet_{x_0}F$ to denote the $\floor{s}$-jet of $F$ at $x_0$, which we identify with its Taylor polynomial
\begin{equation*}
    \jet_{x_0}F(x) \equiv \sum_{0 \leq m < s} \grad^mF(x_0)\brac{\frac{(x-x_0)^{\otimes m}} {m!}} \equiv \sum_{0 \leq \abs{\alpha}< s} \da F(x_0)\frac{(x-x_0)^\alpha}{\alpha!}.
\end{equation*}

A cube $Q$ in $\Rn$ is a set of the form $c_Q + [-\delta_Q/2,\delta_Q/2)^n$, where $c_Q \in \Rn$ is the center of $Q$ and $\delta_Q > 0$ is the sidelength of $Q$. If $A > 0$, we use $AQ$ to denote the concentric dilation of $Q$ by a factor of $A$.
A \underline{dyadic} cube is a cube of the form $2^{-k}z + [-2^{-k},2^{-k})$, where $z \in \mathbb{Z}^n$ and $k \in \mathbb{Z}$. Each dyadic cube $Q$ is contained in a unique dyadic cube with sidelength $2\delta_Q$, and that cube is denoted by $Q^+$.

We use $B(x,r)$ to denote the open ball in $\Rn$ with center $x$ and radius $r$, and we use $Q(c_Q,\delta_Q)$ to denote the cube in $\Rn$ with center $c_Q$ and sidelength $\delta_Q$. 

Let $X$ be a set of parameters, we use $C(X), c(X)$, etc., to denote constants that depend only on $X$. Their precise values may vary from line to line. 

Let $S$ be a finite set, we write $\#S$ to denote the cardinality of $S$. If $ S $ is not finite, we define $ \#S = \infty $. 

Let $ A,B\subset \R^d $. We define $ \diam A := \sup_{x, y \in A}\abs{x-y} $ and $ \dist{A}{B} := \inf_{a \in A, b \in B} \abs{a-b} $.

\section{Basic results on the flat norm}
\label{sect:Fs basic}

We begin by summarizing some basic properties of the class $\Fs$ in the following proposition.

\begin{proposition}\label{prop:Fs basic}
    \newcommand{\Fsp}{\mathcal{F}^{s'}}
    \newcommand{\Fsph}{\dot{\mathcal{F}}^{s'}}
    Let $s >0 $ and let $\Omega\subset \Rn$ be convex. 
    \begin{enumerate}[(A)]
        \item $\Fs(\Omega)$ is a positive cone in $\Cs(\Omega)$, on which $\norm{\cdot}_{\Fs(\Omega)}$ defines a norm. Moreover, $\norm{FG}_{\Fs(\Omega)} \leq C(n,s)\norm{F}_{\Fs(\Omega)}\norm{G}_{\Fs(\Omega)}$ for any $F,G \in \Fs$.
        \item If $0 < s' \leq s$, then $\Fs(\Omega)\subset \Fsp(\Omega)$, and $\norm{\cdot}_{\Fsph(\Omega)} \leq \max\set{\norm{\cdot}_{\Fsh(\Omega)}, \norm{\cdot}_{L^\infty(\Omega)}}$.
        \item If $0 < s' \leq s$, $s > 1$, $F \in \Fs(\Omega)$, and $\inf_{\Omega} F = 0$, then $\norm{F}_{\Fsh(\Omega)} \geq \norm{F}_{L^\infty(\Omega)}$ and $\norm{F}_{\Fsph(\Omega)} \leq \norm{F}_{\Fsh(\Omega)}$.
        \item If $F$ is uniformly bounded away from zero on $\Omega$, then $F \in \Fs(\Omega)$ if and only if $F \in \Cs(\Omega)$.
        \item Assume that $ \Omega $ is bounded.
        Let $F $ be nonnegative and continuously differentiable on $\Omega$, such that $\d_jF \in \Fs(\Omega)$ for every $j = 1, \cdots, n$. Then $F \in \mathcal{F}^{s+1}(\Omega)$ with 
    \begin{equation*}
        \norm{F}_{\mathcal{F}^{s+1}(\Omega)} \leq C(n,s,\diam\Omega)\cdot
        \max
        \set{
        \sum_{j = 1}^n\brac{\norm{\d_j F}_{\mathcal{F}^s(\Omega)}^2}^{1/2},\,
        \norm{F}_{L^\infty(\Omega)}
        }.
    \end{equation*}
    \end{enumerate}
\end{proposition}

To prove Proposition \ref{prop:Fs basic}, we directly apply the univariate argument in the proofs of Theorems 2,3, and 5 of \cite{RSH17} to all the partial derivatives. 
One can also obtain similar wavelet estimates as in Proposition 1 of \cite{RSH17}.
We omit the details here.

\begin{remark}\label{rem:completeness}
    The cone $\Fs$ inherits a notion of completeness from $\Cs$. Suppose $(F_n)_{n =1}^\infty$ is a $\Cs$-Cauchy sequence $\Cs\cap \Fs$ that is also bounded in $\Fs$, then the limit $F$ of $F_n$ in $\Cs$ is in $\Fs$ and $\norm{F_n}_{\Fs} \to \norm{F}_{\Fs}$.
\end{remark}

The following lemma states that the flatness condition (at a single point) induces a lengthscale within which the change in derivatives can be controlled by the pointwise value. In particular, the function is locally constant.  We will use this lemma repeatedly throughout the rest of the paper. 

\begin{lemma}\label{lem:locally constant}
     Given $\epsilon>0$, there exists $c_0>0$ bounded from below by a constant determined only by $\epsilon,n,s$
     such that for any $0 < c \leq c_0$, the following hold.
    \begin{enumerate}[(A)]
        \item Let $x_0 \in \Rn$, let $P \in \P$ with $P(x_0) \geq 0$ and $M := \max\limits_{1 \leq m < s}\brac{\frac{\abs{\grad^m P(x_0)}^s}{P(x_0)^{s-m}}}^{1/m} < \infty$, and let $\delta_{x_0,P}:= c\brac{\frac{P(x_0)}{M}}^{1/s}$. Then
        \begin{equation*}
            \abs{\grad^mP(x) - \grad^mP(x_0)} \leq \epsilon M^{\frac{m}{s}}P(x_0)^{\frac{s-m}{s}}
        \end{equation*}
        for all $x\in B(x_0,\delta_{x_0,P})$ and $0 \leq m < s$.
        In particular, 
        \begin{equation*}
            \abs{P(x) - P(x_0)} \leq \epsilon P(x_0)
            \text{ for all } x\in B(x_0,\delta_{x_0,P}).
        \end{equation*}
        
        \item Let $\Omega \subset \Rn$ be a convex open set, let $x_0 \in \Omega$, $F \in \Fs(\Omega)$, and let $\delta_{x_0,F} := c\brac{\frac{F(x_0)}{\norm{F}_{\Csh(\Omega)} + \norm{F}_{\Fsh(\Omega)}}}^{1/s}$. Then
        \begin{equation*}
            \abs{\grad^mF(x)-\grad^mF(x_0)} \leq 
            \epsilon 
            \brac{\norm{F}_{\Csh(\Omega)} + \norm{F}_{\Fsh(\Omega)}}^{\frac{m}{s}}
            F(x_0)^{\frac{s-m}{s}}
        \end{equation*}
        for all $x \in B(x_0,\delta_{x_0,F})\cap \Omega$ and $0 \leq m < s$.
        In particular,
        \begin{equation*}
            \abs{F(x)-F(x_0)} \leq \epsilon F(x_0)
            \text{ for all } x \in B(x_0,\delta_{x_0,F})\cap \Omega.
        \end{equation*}
        
    \end{enumerate}
\end{lemma}

\begin{proof}
    If $s\leq 1$, there is nothing to prove. We assume that $s > 1$. We further assume that $P(x_0), F(x_0) > 0$, for otherwise, all derivatives of $P$ and $F$ must vanish, and the estimates are trivial.

    First, we prove (A). Write $\delta = \delta_{x_0,P}$. Let $x \in B(x_0,c\delta)$. We treat $P$ as its own Taylor polynomial, so that
    \begin{equation*}
        \grad^m P(x) = \grad^m P(x_0) + \sum_{1 \leq k < s-m}\grad^{m+k} P(x_0) \brac{\frac{(x-x_0)^{\otimes k}}{k!}}.
    \end{equation*}
    Note that $\abs{(x-x_0)^{\otimes k}} \leq c^k\delta^{k}$.
    By assumption, 
    $\abs{\grad^{m+k} P(x_0)} \leq P(x_0)^{\frac{s-m-k}{s}} M^{\frac{m+k}{s}} $. Therefore, 
    \begin{equation*}
        \begin{split}
            \abs{\grad^mP(x)-\grad^mP(x_0)} \leq C(n,s)\sum_{1 \leq k < s-m} P(x_0)^{\frac{s-m-k}{s}} M^{\frac{m+k}{s}} \cdot c^k \cdot \frac{P(x_0)^{\frac{k}{s}}}{M^{\frac{k}{s}}}.
        \end{split}
    \end{equation*}
    For $c_0$ sufficiently small, we see that the right-hand side can be bounded by $\epsilon M^{\frac{m}{s}}P(x_0)^{\frac{s-m}{s}}$. This proves (A).

    We turn to (B). We set write $\delta = \delta_{x_0,F}$. Let $P:\equiv \jet_{x_0}F$, i.e.,
    \begin{equation*}
        P(x) := \sum_{0\leq m < s}\grad^m F(x_0) \brac{\frac{(x-x_0)^{\otimes m}}{m!}}.
    \end{equation*}
    Then $\norm{F}_{\Fsh(\Omega)} \geq M := \max\limits_{1 \leq m < s}\brac{\frac{\abs{\grad^m P(x_0)}^s}{P(x_0)^{s-m}}}^{1/m}$. By the triangle inequality, Taylor's theorem, and part (A) (with $\epsilon/2$ in place of $\epsilon$), 
    \begin{equation*}
        \begin{split}
            \abs{\grad^mF(x) - \grad^mF(x_0)} &\leq \abs{\grad^mP(x)-\grad^mP(x_0)} + \abs{\grad^m(F-P)(x)} \\
            &\leq 
            \frac{\epsilon}{2}\norm{F}_{\Fsh(\Omega)}^{\frac{m}{s}}F(x_0)^{\frac{s-m}{s}}
            + \abs{\grad^m(F-P)(x)}.
        \end{split}
    \end{equation*}
    By Taylor's theorem,
    \begin{equation*}
        \begin{split}
            \abs{\grad^m(F-P)(x)} 
            &\leq C(n,s)\norm{F}_{\Csh(\Omega)}\delta^{s-m} \\
            &\leq C(n,s)c^{s-m}\cdot \brac{\norm{F}_{\Csh(\Omega)}
            +
            \norm{F}_{\Fsh(\Omega)}
            }
            \cdot
            \brac{\frac{F(x_0)}{\norm{F}_{\Csh(\Omega)}
            +
            \norm{F}_{\Fsh(\Omega)}
            }}^{\frac{s-m}{s}}
            \\
            &\leq \frac{\epsilon}{2}\brac{\norm{F}_{\Csh(\Omega)}
            +
            \norm{F}_{\Fsh(\Omega)}
            }^{\frac{m}{s}}
            F(x_0)^{\frac{s-m}{s}}
        \end{split}
    \end{equation*}
    as long as $c_0$ is sufficiently small. Lemma \ref{lem:locally constant}(B) follows from the two inequalities above.
\end{proof}

If $s \in (0,1]$, then the flatness condition is trivial, and any function in $ \Fs(\Omega)$ can be extended to a nonnegative $\Cs(\Rn)$ function with comparable norm, via Kirszbraun's formula or the Whitney extension operator (see Section \ref{sect:whitney} below). The next proposition, which improves Theorem 4 of \cite{RSH17}, says a similar phenomenon also occurs for $s \in (1,2]$. The scaling is consistent with the nonnegative (non-H\"older) $C^2$ extension in \cite{JL20,JL20-Alg,JL20-Ext,FJL23,JLLL23}. Results of such type can be also used to study nonnegative Sobolev $L^2_p(\Rn)$ ($p > n$) extension, thanks to the embedding $L^2_p(\Omega) \hookrightarrow \mathcal{C}^{2-n/p}(\Omega) $ for bounded $\Omega\subset\Rn$. See \cite{I13}.

\begin{proposition}\label{prop:c2}
    Let $s \in (1,2]$, let $x_0 \in \Rn$, and let $P$ be an affine polynomial with $P(x_0) \geq 0$ and $\frac{\abs{\grad P}^s}{P(x_0)^{1-s}} < \infty$. Set $\delta= \frac{P(x_0)}{\abs{\grad P}}$. Then for any $\lambda > 1$ and nonnegative $F \in \Cs(B(x_0,\lambda\delta))$ with $\jet_{x_0}F \equiv P$, we have
    \begin{equation}
    \norm{F}_{\Csh(B(x_0,\lambda\delta))} \geq
    \frac{s(\lambda-1)}{\lambda^{s}}\frac{\abs{\grad P}^s}{P(x_0)^{s-1}}.
    \label{eq:prop:c2-0}
    \end{equation}
    In particular, if $F \in \Cs(\Rn)$ is nonnegative and $F_\Omega$ is the restriction of $F$ to $\Omega \subset \Rn$, then $F_\Omega \in \Fs(\Omega)$ and $\norm{F_\Omega}_{\Fsh(\Omega)} \leq \frac{2^s}{s}\norm{F}_{\Csh(\Rn)}$.
\end{proposition}

\begin{proof}
    If $\delta = 0$, then $P(x_0) = \abs{\grad P} = 0$, and \eqref{eq:prop:c2-0} is trivial. We assume that $\delta > 0$.
    Let $M := \norm{F}_{\Cs(B(x_0,\lambda\delta))}$. For any $x \in B(x_0,\lambda\delta)$, let $\gamma$ be the straight line segment from $x_0$ to $x$, and the fundamental theorem of calculus gives
    \begin{equation*}
        F(x) = P(x) + \int_\gamma (\grad F - \grad P).
    \end{equation*}
    The triangle inequality then implies that
    \begin{equation}
        \abs{F(x) - P(x)} \leq s^{-1}\abs{x-x_0}^sM.
        \label{eq:prop:c2-1}
    \end{equation}
    Restricting $F$ to the ray emanating from $x_0$ in the direction of $-\grad P$ and re-parameterize, we see that \eqref{eq:prop:c2-1} implies
    \begin{equation}
        s^{-1}M t^{s} - \abs{\grad P}t + P(x_0) \geq 0
        \text{ for all } t \in [0,\lambda\delta].
        \label{eq:prop:c2-2}
    \end{equation}
    Plugging $t = \lambda\delta$ into \eqref{eq:prop:c2-2}, we see that \eqref{eq:prop:c2-0} follows.

    For the last conclusion, since $ F $ is assumed to be globally nonnegative, we may set
     $\lambda = 2$ and apply \eqref{eq:prop:c2-2} pointwise everywhere in $\Omega$. 
\end{proof}

\section{Whitney jets}
\label{sect:whitney}

\begin{definition}
    Let $E\subset \Rn$ be a closed set. A \underline{Whitney field} on $E$ is a parametrized family of polynomials $\Vec{P} = (P_x)_{x\in E}$ such that $P_x \in \P$ for all $x \in E$. We define
    \begin{equation}
        \begin{split}
            \norm{\Vec{P}}_{\JC(E)}&:= \sup_{\substack{x \in E\\0 \leq m < s}}\abs{\grad^mP_x(x)} + \sup_{\substack{x,y \in E\\x \neq y\\0 \leq m < s}}\frac{\abs{\grad^m(P_x-P_y)(y)}}{\abs{x-y}^{s-m}}, \\
            \norm{\Vec{P}}_{\JFh(E)}&:= \sup_{\substack{x \in E\\ 1 \leq m < s}}\brac{\frac{\abs{\grad^m P_x(x)}^s}{\abs{P_x(x)}^{s-m}}}^{1/m}, \text{ and }
            \\
            \norm{\vec{P}}_{\JF(E)}&:= \norm{\Vec{P}}_{\JC(E)} + \norm{\Vec{P}}_{\JFh(E)}.
        \end{split}
        \label{eq:Whitney norm}
    \end{equation}
    
    We use $\JF(E)$ to denote the positive cone of Whitney fields $\Vec{P} = (P_x)_{x\in E}$ such that $P_x(x) \geq 0$ for all $x \in E$ and $\norm{\Vec{P}}_{\JF(E)}<\infty$.
\end{definition}

Note that if $E\subset \Rn$ is any set and $F \in \Fs(\Rn)$, then $F$ generates a Whitney field
\begin{equation*}
    \jet_E F :\equiv (\jet_{x}F)_{x \in E} \in \JF(E).
\end{equation*}

\begin{theorem}[Taylor's theorem]\label{thm:Taylor}
    Let $E\subset \Rn$ be a closed set and let $F \in \Fs(\Rn)$. Then $ \jet_EF \in \JF(E) $ and
     $\norm{\jet_E F}_{\JF(E)}\leq C(n,s)\norm{F}_{\Fs(\Rn)}$. 
\end{theorem}

The Whitney extension map then is a continuous right inverse of $\jet_E : \Fs(\Rn) \to \JF(E)$. See Theorems \ref{thm:whitney-finite} and \ref{thm:whitney-closed} below for the precise statement.

\subsection{A model bump function}
\label{sect:pou}

We fix a model bump function for the rest of the paper.

\begin{proposition}\label{prop:bump-flat}
Consider the standard bump function $g : \R \to \pos$ defined by
\begin{equation}
    g(x) = 
    \begin{cases}
    \exp\brac{1-\frac{1}{1-t^2}} &\text{ for } -1 < t < 1
    \\
    0 &\text{ otherwise}
    \end{cases}
    .
    \label{bump-std}
\end{equation}
Then for all $s>0$, we have
\begin{equation*}
    \norm{g}_{\Fs(\R)} \leq C (n,s).
\end{equation*}
\end{proposition}

\begin{proof}
\newcommand{\bzero}{\mathbf{0}}
It suffices to check $\norm{g}_{\Fsh([-1,1])}$. For $t \in (-1,1)$, we have $  g^{(m)}(t) = q_m(t)g(t)$ where each term of $q_m$ is a rational function, singular only at $\pm 1$. Then 
\begin{equation*}
    \brac{\abs{g^{(m)}(t)}^s/g(t)^{s-m}}^{1/m} = \abs{q_m(t)}^{s/m} \cdot g(t) \leq C(s,m)
\end{equation*}
since $g$ vanishes faster than any polynomial near $\pm 1$. 
\end{proof}

With $g$ as in \eqref{bump-std}, we define the standard bump function $\phi_0:\Rn \to \pos$ by
\begin{equation}
    \phi_0 := g^{\otimes n},\quad 
    (t_1, \cdots, t_n) \mapsto \prod_{j = 1}^n g(t_j).
    \label{eq:phi_0-def}
\end{equation}
 It is clear that 
 \begin{equation*}
     \supp{\phi_0}= [-1,1]^n
     \text{ and }
     \norm{\phi_0}_{\Fs(\Rn)} \leq C(n,s).
 \end{equation*}

Let $\delta > 0$ and let $\tau_\delta: x \mapsto \delta^{-1} x$. A direct computation yields
\begin{equation}
        \norm{\phi\circ \tau_\delta}_{\Fsh(\Rn)} = \delta^{-s}\norm{\phi}_{\Fsh(\Rn)}
        \text{ for any } \phi \in \Fs(\Rn).
        \label{eq:rescaling-delta}
\end{equation}

\subsection{Extension of a single jet}

For $x_0 \in \Rn$ and $M \geq 0$, we define
\begin{equation}
    \Gamma(x_0,M):= \set{P \in \P: P(x_0)\geq 0,\,
    \max_{0 \leq m < s}\abs{\grad^m P(x_0)} \leq M
    \text{ and }
    \max\limits_{1 \leq m < s}\brac{\frac{\abs{\grad^m P(x_0)}^s}{P(x_0)^{s-m}}}^{1/m} \leq M
    .
    }
    \label{eq:Gamma-def}
\end{equation}
and
\begin{equation}
    \Gamma(x_0) := \bigcup_{M\geq 0}\Gamma(x_0,M).
\end{equation}

\begin{lemma}\label{lem:Whitney-single jet}
For every $x_0 \in \Rn$, there exists a map $\T_{x_0} : \Gamma(x_0)\to \Fs(\Rn)$ such that for every $P \in \Gamma(x_0,M)$, the following hold.
\begin{enumerate}[(A)]
    \item $\jet_{x_0}\circ \T_{x_0}[P] \equiv P$.
    \item $\T_{x_0}[P]\geq 0$.
    \item $\norm{\T_{x_0}[P]}_{\Fsh(\Rn)}\leq C(n,s)M$.
    
    \item $\norm{\T_{x_0}[P]}_{\Cs(\Rn)}\leq C(n,s)M$, and consequently, $\norm{\T_{x_0}[P]}_{\Fs(\Rn)} \leq C(n,s)M$.
\end{enumerate}
\end{lemma}

\begin{proof}
    Fix $P \in \Gamma(x_0)$ and let $M := \max\limits_{1 \leq m < s}\brac{\frac{\abs{\grad^m P(x_0)}^s}{P(x_0)^{s-m}}}^{1/m}$. If $P(x_0) = 0$, then $P$ is the zero polynomial and $M=0$, so we can set $\T_{x_0}[P] \equiv 0$. From now on, we assume $P(x_0) > 0$. We may further assume that $M = 1$.
    
    We set 
    \begin{equation*}
        \delta:= P(x_0)^{1/s}
        \text{ so that } P(x_0) = \delta^{s}.
    \end{equation*}
    By the definition of $\Gamma$, we have
    \begin{equation}
        \abs{\grad^m P(x_0)} \leq \delta^{s-m}
        \text{ for } 1 \leq m < s.
        \label{eq:lem:Whitney-single-Pflat}
    \end{equation}

    Let $c_0$ be as in Lemma \ref{lem:locally constant} with $\epsilon = 1/2$, and let $Q := x_0 + [-c_0\delta,c_0\delta] \subset \Rn$.  Thus,
    \begin{equation}
        P(x) \geq 1/2\delta^{s}
        \text{ for all }x \in Q.
        \label{eq:lem:Whitney-lower bound}
    \end{equation}
    From \eqref{eq:lem:Whitney-single-Pflat} and Taylor's theorem, we also see that
    \begin{equation}
        \abs{\grad^m P(x)} \leq \delta^{s-m}C(n,s)
        \text{ for all $x \in Q$ and $0 \leq m < s$.}
        \label{eq:lem:Whitney-single-Pflat2}
    \end{equation}

    Combining  \eqref{eq:lem:Whitney-lower bound} and \eqref{eq:lem:Whitney-single-Pflat2}, we see that
    \begin{equation}
        \brac{\frac{\abs{\grad^m P(x)}^s}{P(x)^{s-m}}}^{1/m} \leq C(n,s)
        \text{ for all $ x\in Q$ and $1 \leq m < s$.}
        \label{eq:lem:Whitney-single-Pflat3}
    \end{equation}

    \paragraph{Case I: suppose $\delta \geq 1$.}

     We let $\theta_0 := \phi_0(\frac{x-x_0}{c_0})$ with $\phi_0$ as in \eqref{eq:phi_0-def}. Define
    \begin{equation}
        \T_{x_0}[P]:= \theta_0\cdot P.
    \end{equation}
    Since $\theta_0 \equiv 1$ near $x_0$, conclusion (A) follows. Since $\supp{\theta_0}\subset x_0 + [-c_0, c_0]^n\subset Q$, conclusion (B) follows from \eqref{eq:lem:Whitney-lower bound}. Thanks to Proposition \ref{prop:Fs basic}(A) and \eqref{eq:lem:Whitney-single-Pflat3}, we have
    \begin{equation*}
        \norm{\theta_0 P}_{\Fsh(Q)} \leq C(n,s)\norm{\theta_0}_{\Fsh(Q)}\norm{P}_{\Fsh(Q)} \leq C(n,s).
    \end{equation*}
    Conclusion (C) follows. Conclusion (D) is similar since $\Cs$ is an algebra. This concludes the analysis of Case I.
    
    \paragraph{Case II: suppose $\delta < 1$.}
    Let $\theta(x) := \phi_0(\frac{x-x_0}{c_0\delta})$ with $\phi_0$ as in \eqref{eq:phi_0-def}. Then
    \begin{enumerate}[($\theta$-1)]
        \item $\theta\equiv 1$ near $x_0$,
        \item $\supp{\theta}\subset Q$, $0\leq \theta \leq 1$,
        \item $\abs{\grad^m\theta}\leq \delta^{-m}\cdot C(n,s)$,
        \item and $\norm{\theta}_{\Fsh(Q)} \leq \delta^{-s}\cdot C(n,s)$, thanks to \eqref{eq:rescaling-delta}.
    \end{enumerate}
    
    We define
    \begin{equation}
        \T_{x_0}[P] := \theta \cdot P.
    \end{equation}

    Conclusion (A) follows from property ($\theta$-1).

    Conclusion (B) follows from \eqref{eq:lem:Whitney-lower bound} and properties ($\theta$-2) and ($\theta$-3).

Now we prove conclusion (C). Write $F := \T_{x_0}[P]$. For $1 \leq m < s$ and $x \in Q$,
\begin{equation}
        \grad^mF(x) = \sum_{k + l = m} C(k,l) \grad^k\theta(x)\otimes \grad^lP(x)
        \label{eq:lem:Whitney:derivatives}
\end{equation}
For each summand, we have
\begin{equation}
    \begin{split}
        \brac{\frac{\abs{\grad^k\theta(x)\otimes \grad^l P(x)}^s}{\theta(x)^{s-m} P(x)^{s-m}}}^{1/m} 
        &\leq C(n,s)\brac{\theta(x)^l P(x)^k \frac{\abs{\grad^k\theta(x)}^s}{\theta(x)^{s-k}}
        \frac{\abs{\grad^l P(x)}^s}{P(x)^{s-l}}
        }^{1/m}\\
        &\leq C(n,s) \theta(x)^{l/m}P(x)^{k/m}\norm{\theta}_{\Fsh(Q)}^{k/m}\norm{P}_{\Fsh(Q)}^{l/m} \\
        \text{by ($\theta$-4) and \eqref{eq:lem:Whitney-single-Pflat3}}&\leq C(n,s) \delta^{sk/m}\delta^{-sk/m} \leq C(n,s). 
    \end{split}
    \label{eq:lem:Whitney:summand}
\end{equation}
Combine \eqref{eq:lem:Whitney:derivatives} and \eqref{eq:lem:Whitney:summand}, we have
\begin{equation*}
    \brac{\frac{\abs{\grad^m F(x)}^s}{F(x)^{s-m}}}^{1/m}
    \leq C(n,s)\sum_{k+l = m}\brac{\frac{\abs{\grad^k\theta(x)\otimes \grad^l P(x)}^s}{\theta(x)^{s-m} P(x)^{s-m}}}^{1/m}  \leq C(n,s).
    \label{eq:lem:Whitney:derivatives-done}
\end{equation*}
Since \eqref{eq:lem:Whitney:derivatives} holds for all $x \in Q$ and $1 \leq m < s$, conclusion (C) follows.

We turn to conclusion (D).

For $x,y \in Q$ and $k,l$ with $k+l = m = \floor{s}$,
\begin{equation}
    \begin{split}
        &\abs{\grad^k\theta(x)\otimes \grad^lP(x)- \grad^k\theta(y)\otimes \grad^lP(y)} \leq 
        \\
        &\quad \quad \quad \abs{\grad^k\theta(x)\otimes (\grad^l P(x)-\grad^lP(y))} + \abs{(\grad^k\theta(x)-\grad^k\theta(y))\otimes \grad^lP(y)}=: A_1 + A_2
    \end{split}
    \label{eq:lem:Whitney-Holder1}
\end{equation}
Thanks to \eqref{eq:lem:Whitney-single-Pflat2} and ($\theta$-4), we have
\begin{equation}
    \begin{split}
        A_1 &\leq C(n,s)\norm{\grad^k\theta}_{L^\infty(Q)} \norm{\grad^{l+1}P}_{L^\infty(Q)}\abs{x-y} \\
        &\leq C(n,s)\delta^{-k}\delta^{s-l-1}\abs{x-y}\\
        &\leq C(n,s)\abs{x-y}^{s-m}.
    \end{split}
    \label{eq:lem:Whitney-Holder2}
\end{equation}
Similarly, 
\begin{equation}
    \begin{split}
        A_2 &\leq C(n,s)\norm{\grad^{k+1}\theta}_{L^\infty(Q)}\abs{x-y}\norm{\grad^l P}_{L^\infty(Q)} \\
        &\leq C(n,s)\delta^{-(k+1)}\delta^{s-l}\abs{x-y}\\
        &\leq C(n,s)\abs{x-y}^{s-m}.
    \end{split}
    \label{eq:lem:Whitney-Holder3}
\end{equation}
It follows from \eqref{eq:lem:Whitney:derivatives}, \eqref{eq:lem:Whitney-Holder1}--\eqref{eq:lem:Whitney-Holder3} that
\begin{equation*}
    \norm{F}_{\Csh(\Rn)}\leq C(n,s).
\end{equation*}

Finally, using the assumption that $\abs{\grad^mP(x_0)} \leq 1$ for $0 \leq m < s$, we can apply \eqref{eq:lem:Whitney-single-Pflat2} and ($\theta$-4) to estimate \eqref{eq:lem:Whitney:derivatives} to obtain
\begin{equation*}
    \abs{\grad^m F(x)} \leq C(n,s)
    \text{ for all $x \in Q$ and $0 \leq m < s$}.
\end{equation*}
Conclusion (D) follows.

\end{proof}

\subsection{Whitney Extension Theorem for finite sets}

Let $E\subset \Rn$ be a finite set. A finite \underline{Whitney decomposition of $\Rn$ with respect to $E$} is a collection of dyadic cubes $\Lambda_E$ defined by
\begin{equation*}
    \Lambda_E := \bigcup_{Q \in \mathbb{Z}^n}\Lambda_{Q}
    \text{ where } \Lambda_{Q} := \begin{cases}
        \set{Q} &\text{ if } \#(E \cap 3Q) \leq 1\\
        \bigcup\set{\Lambda_{Q'}: \text{ $(Q')^+ = Q$}} &\text{ otherwise}
    \end{cases}.
\end{equation*}

Given $Q,Q' \in \Lambda_E$, we say that $Q$ touches $Q'$ if their closures have nonempty intersection.
The cubes in $\Lambda$ satisfy a nice geometric property:
\begin{equation}
    1/2\delta_Q \leq \delta_{Q'} \leq 2\delta_Q
    \text{ whenever $Q$ touches $Q'$}.
    \label{eq:comparable lengthscale}
\end{equation}
As a consequence,
\begin{equation}
    \#\set{Q\in\Lambda_E : 1.1Q \ni x} \leq 2^n
    \text{ for all } x \in \Rn.
    \label{eq:bounded overlap}
\end{equation}

For each $Q \in \Lambda_E$, we set
\begin{equation*}
    \phi_Q(x) := \phi_0\brac{2(x-c_Q)/(1.1\delta_Q)}
\end{equation*}
with $ \phi_0 $ as in \eqref{bump-std}. 
Then 
\begin{equation*}
    \supp{\phi_Q} = 1.1 Q
    ,\,
    \abs{\grad^k\phi_Q(x)}\leq \delta_Q^{-k}\cdot C(n,k), \text{ and }
    \norm{\phi_Q}_{\Fsh(\Rn)} \leq \delta_Q^{-s}\cdot C(n,s)
    \text{ for } 0 \leq k < s.
\end{equation*}
We define
\begin{equation*}
    \theta_Q(x) := \frac{\phi_{1.1Q}(x)}{\sum_{Q\in\Lambda_E}\phi_{1.1Q}(x)}.
\end{equation*}
The denominator is a finite sum, thanks to \eqref{eq:bounded overlap}.
$\set{\theta_Q: Q \in \Lambda_E}$ has the following properties.
\begin{enumerate}[(POU-1)]
    \item $0 \leq \theta_Q \leq 1$, and $\sum_{Q\in\Lambda}\theta_Q \equiv 1$;
    \item $\supp{\theta_Q}\subset 1.1Q$ for each $Q \in \Lambda_E$, so each $x \in \Rn$ lies in the support of at most $2^n$ partition functions; 
    \item $\abs{\grad^m\theta_Q}\leq \delta_Q^{-m}\cdot C(n,s)$ for $1 \leq m <s$;
    \item $\norm{\theta_Q}_{\Fsh(\Rn)} \leq \delta^{-s}\cdot C(n,s)$. 
\end{enumerate}

For each $Q \in \Lambda_E$, we define a map $\T_Q $ as follows.
\begin{itemize}
    \item Suppose $E\cap 3Q \neq \void$, we pick $x_Q \in E\cap 3Q$ and set $\T_Q := \T_{x_Q}$, with $\T_{x_Q}$ as in Lemma \ref{lem:Whitney-single jet}. We set $P_Q \equiv P_{x_Q}$.
    \item Suppose $E\cap 3Q = \void$ and $\delta_Q < 1$. Then $E \cap 3Q^+ \neq \void$. We fix any $x_Q \in E\cap 3Q^+$ and set $\T_Q := \T_{x_Q}$, with $\T_{x_Q}$ as in Lemma \ref{lem:Whitney-single jet}. We set $P_Q \equiv P_{x_Q}$.
    \item Suppose $E\cap 3Q = \void$ and $\delta_Q = 1$. We set $\T_Q := $ the zero map. We set $P_Q \equiv 0$.
\end{itemize}
Define $\T_E : \JF(E) \to \Fs(\Rn)$ by
\begin{equation}
    \T_E[\Vec{P}](x):= \sum_{Q\in \Lambda_E}\theta_Q(x)\cdot \T_{x_Q}[P_Q](x).
    \label{eq:whitney-finite-op}
\end{equation}

\begin{theorem}\label{thm:whitney-finite}
    Let $E \subset \Rn$ be a finite set. The map $\T_E : \JF(E) \to \Fs(\Rn)$ defined in \eqref{eq:whitney-finite-op} satisfies the following properties.
    \begin{enumerate}[(A)]
        \item $\jet_E\circ\T_E[\Vec{P}]\equiv \Vec{P} $ for all $\Vec{P}\in \JF(E)$, 
        \item $\norm{\T_E[\Vec{P}]}_{\Cs(\Rn)}\leq C(n,s)\norm{\Vec{P}}_{\JC(E)}$,  and
        \item $\norm{\T_E [\Vec{P}]}_{\Fsh(\Rn)} \leq C(n,s)\norm{\Vec{P}}_{\JF(E)}$, and hence, $\norm{\T_E [\Vec{P}]}_{\Fs(\Rn)} \leq C(n,s)\norm{\Vec{P}}_{\JF(E)}$.
    \end{enumerate}
\end{theorem}

\begin{proof}

Let $ F := \T_E[\vec{P}] $ and let $F_Q := \T_{x_Q}[P_Q]$ for $ Q \in\Lambda_E $. 

Since $F$ is defined to be the convex combination of nonnegative functions, we have $F \geq 0$.

Conclusion (A) follows from Lemma \ref{lem:Whitney-single jet}(A) and the fact that $\sum_{Q\in\Lambda_E}\theta_Q \equiv 1$.  

Conclusion (B) is almost identical to the proof of the classical Whitney Extension Theorem for $ \Cs(\Rn) $ \cite{G58}. The first minor difference is that we use a finite cutoff of the Whitney decomposition. Secondly,
we glue together $F_Q $ instead of $P_Q$ to preserve nonnegativity, which is harmless thanks to Taylor's theorem.





Now we prove conclusion (C). Note that there is nothing to prove if $s \leq 1$. For the rest of the proof, we assume $s > 1$. We set
\begin{equation*}
    M_1 := \norm{\Vec{P}}_{\JC(E)}
    \text{ and }
    M_2 := \norm{\Vec{P}}_{\JFh(E)}.
\end{equation*}

 For $1 \leq m < s$, $x \in \Rn$, $Q(x) \in \Lambda_E$ containing $x$, the product rule yields
\begin{equation}
    \grad^m F(x) = \sum_{Q\in\Lambda_E}\theta_{Q}(x)\grad^mF_{Q}(x) + \sum_{Q \text{ touches } Q(x)}\sum_{\substack{k + l = m\\k\geq 1}}C(k,l) \grad^k\theta_{Q}(x)\otimes \grad^l (F_{Q}-F_{Q(x)})(x).
    \label{eq:thm:whitney:derivatives}
\end{equation}
Note that in \eqref{eq:thm:whitney:derivatives}, every sum has bounded number of summands, thanks to (POU-2) and \eqref{eq:bounded overlap}. We will repeatedly use this fact without further referencing it.

Using \eqref{eq:thm:whitney:derivatives}, we see that to estimate $\norm{F}_{\Fsh}$, it suffices to estimate the following two types of terms:
\begin{equation*}
    A_{0,m}:= \brac{\frac{\abs{\theta_Q(x)\grad^m F_Q(x)}^{s}}{F(x)^{s-m}}}^{1/m}
    \text{ and }
    A_{k,l}:=\brac{\frac{\abs{\grad^k\theta_{Q}(x)\otimes \grad^l(F_{Q}-F_{Q(x)})(x)}^s}{F(x)^{s-m}}}^{1/m},
\end{equation*}
with the understanding that $Q$ touches $Q(x)$ and $k+l = m$.

For the first type, we have
\begin{equation}
    A_{0,m}\leq \theta_Q(x)^{s/m}\brac{\frac{\abs{\grad^mF_{Q}(x)}^s}{F_{Q}(x)^{s-m}} }^{1/m} \leq \norm{F_Q}_{\Fsh(\Rn)} \leq M_2.
    \label{eq:epsilon-0m}
\end{equation}


For the analysis of $A_{k,l}$, first we estimate $\abs{\grad^l(F_Q - F_{Q(x)})(x)}$. By the triangle inequality,
\begin{equation}
    \begin{split}
        \abs{\grad^l(F_{Q} - F_{Q(x)})(x)} &\leq \abs{\grad^l (F_{Q} - P_{Q})(x)} + \abs{\grad^l(F_{Q(x)} - P_{Q(x)})(x)} + \abs{\grad^l(P_{Q}-P_{Q(x)})(x)}.
    \end{split}
    \label{eq:thm:whitney:123}
\end{equation}
Thanks to Lemma \ref{lem:Whitney-single jet}(B) and Taylor's theorem, we have
\begin{equation}
    \begin{split}
        \abs{\grad^l (F_{Q} - P_{Q})(x)} &\leq \abs{x-x_{Q}}^{s-l}\cdot C(n,s)\norm{F_{Q}}_{\Cs(\Rn)} \leq \delta^{s-l}\cdot C(n,s)M_1, \text{ and }\\
        \abs{\grad^l (F_{Q(x)} - P_{Q(x)})(x)} &\leq \abs{x-x_{Q(x)}}^{s-l}\cdot C(n,s)\norm{F_{Q(x)}}_{\Cs(\Rn)} \leq \delta^{s-l}\cdot C(n,s)M_1.
    \end{split}
    \label{eq:thm:whitney:12}
\end{equation} 
Recall that $x_Q \in E\cap 3Q^+$ for every $Q \in \Lambda_E$. Thanks to \eqref{eq:comparable lengthscale} and Taylor's theorem, we have
\begin{equation}
    \abs{\grad^l(P_{Q}-P_{Q(x)})(x)}\leq
    \abs{x_{Q}-x_{Q(x)}}^{s-l}\cdot C(n,s)M_1
    \leq 
    \delta^{s-l}\cdot C(n,s)M_1.
    \label{eq:thm:whitney:3}
\end{equation}
Combining \eqref{eq:thm:whitney:123}--\eqref{eq:thm:whitney:3} above, we have
\begin{equation}
    \abs{\grad^l(F_{Q}-F_{Q(x)})(x)} \leq \delta^{s-l}\cdot C(n,s)M_1
    .
    \label{eq:thm:whitney:neighbor estimate}
\end{equation}

We then consider two cases. 

\paragraph{Case I: $F(x)\geq \delta^s\cdot M_2$.}

Thanks to \eqref{eq:thm:whitney:neighbor estimate} and (POU-3), we have
\begin{equation}
    \begin{split}
        A_{k,l} \leq  \brac{\delta^{-ks}\delta^{(s-l)s}\delta^{-(s-m)s}}^{1/m} C(n,s)(M_1 + M_2) = C(n,s)(M_1 + M_2).
    \end{split}
    \label{eq:epsilon-kl-1}
\end{equation}

\paragraph{Case II: $F(x) < \delta^s\cdot M_2$}

This implies that 
\begin{equation}
    F_Q(x) \leq \delta^s\cdot M_2
    \text{ for every $Q\in \Lambda_E$ with $1.1Q\ni x$.}
    \label{eq:FQ-small}
\end{equation}
Since $\norm{F_Q}_{\Fsh(\Rn)} \leq C(n,s)M_2$ by Lemma \ref{lem:Whitney-single jet}, we see that
\begin{equation*}
    \abs{\grad^l F_Q(x)} \leq \delta^{s-l}\cdot C(n,s)M_2.
\end{equation*}
Therefore, 
\begin{equation}
    \begin{split}
        \brac{\frac{\abs{
    \grad^k\theta_Q(x)\otimes \grad^l F_Q(x)
    }^s}{F(x)^{s-m}}}^{1/m}
    &\leq \brac{\frac{\abs{
    \grad^k\theta_Q(x)\otimes \grad^l F_Q(x)
    }^s}{\theta_Q^{s-m}F_Q(x)^{s-m}}}^{1/m}\\
    &\leq C(n,s)\brac{ \theta_Q(x)^lF_Q(x)^k \frac{\abs{\grad^k\theta_Q(x)}^s}{\theta_Q(x)^{s-m}} \frac{\abs{\grad^l F_Q(x)}^s}{F_Q(x)^{s-l}} }^{1/m}
    \\
   \text{\eqref{eq:FQ-small} and (POU-3)} &\leq C(n,s)\theta_Q(x)^{l/m}F_Q(x)^{k/m}\norm{\theta_Q}_{\Fsh(\Rn)}^{k/m}\norm{F_Q}_{\Fsh(\Rn)}^{l/m}\\
    &\leq \delta^{k/m}\delta^{-k/m}\cdot C(n,s)M_2 = C(n,s)M_2.
    \end{split}
    \label{eq:FQ-small-flat}
\end{equation}
Since $Q$ touches $Q(x)$, $\abs{\grad^k\theta_Q(x)}\leq C(n,s)\theta_{Q(x)}$. By a similar argument as in \eqref{eq:FQ-small-flat}, we have
\begin{equation}
    \brac{\frac{\abs{
    \grad^k\theta_Q(x)\otimes \grad^l F_{Q(x)}(x)
    }^s}{F(x)^{s-m}}}^{1/m} \leq C(n,s)M_2.
    \label{eq:FQ-small-flat2}
\end{equation}
Combining \eqref{eq:FQ-small-flat} and \eqref{eq:FQ-small-flat2}, we can conclude that
\begin{equation}
    A_{k,l} \leq C(n,s)M_2.
    \label{eq:epsilon-kl-2}
\end{equation}

Thanks to \eqref{eq:epsilon-0m}, \eqref{eq:epsilon-kl-1}, and \eqref{eq:epsilon-kl-2}, we have
\begin{equation}
    \begin{split}
    \brac{\frac{\abs{\grad^mF(x)}^s}{F(x)^{s-m}}}^{1/m}
    &\leq C(n,s) \sum_{Q\in\Lambda_E} \brac{\frac{\abs{\theta_Q(x)\grad^m F_Q(x)}^{s}}{F(x)^{s-m}}}^{1/m}
    \\
    &\quad \quad \quad + C(n,s)\sum_{\substack{Q\text{ touches }Q(x)\\k + l = m\\k \geq 1}}
    \brac{\frac{\abs{\grad^k\theta_{Q}(x)\otimes \grad^l(F_{Q}-F_{Q(x)})(x)}^s}{F(x)^{s-m}}}^{1/m}\\
    &\leq C(n,s)M_2.
    \end{split}
    \label{eq:conc-c}
\end{equation}
Since \eqref{eq:conc-c} holds for all $x \in \Rn$ and $0 \leq m < s$, we see that conclusion (C) follows. This concludes the proof of Theorem \ref{thm:whitney-finite}.
\end{proof}

\subsection{Whitney Extension Theorem for arbitrary closed sets}

For this section, all the dyadic cubes are assumed to be closed.

Let $E\subset \Rn$ be a closed set (not necessarily finite).  Let $\Lambda_E^\infty$ be the classical Whitney decomposition of $\Rn\setminus E$ defined by
\begin{equation*}
    \Lambda_E^\infty := \bigcup_{Q \in \mathbb{Z}^n}\Lambda_{Q}^\infty
    \text{ where } \Lambda_{Q}^\infty := \begin{cases}
        \set{Q} &\text{ if } E \cap 3Q =\void \\
        \bigcup\set{\Lambda_{Q'}^\infty: \text{ $(Q')^+ = Q$}} &\text{ otherwise}
    \end{cases}.
\end{equation*}
As before,
\begin{equation*}
    1/2\delta_Q\leq \delta_{Q'} \leq 2\delta_Q
    \text{ whenever $Q, Q' \in \Lambda_E^\infty$ touch each other.}
\end{equation*}
Thus,
cubes in $\Lambda_E^\infty$ have bounded overlap, and the partition of unity $\set{\theta_Q:Q\in\Lambda_E^\infty}$ adapted to $\Lambda_E^\infty$ satisfies (POU-1)--(POU-4).

For each $Q \in \Lambda_E^\infty$, we pick $x_Q \in E$ (not necessarily unique) such that
\begin{equation*}
    \dist{x_Q}{E} = \dist{E}{Q}.
\end{equation*}
For $\Vec{P} = (P_y)_{y\in E} \in \JF(E)$ We define
\begin{equation}
    \T_E[\Vec{P}](x) = \begin{cases}
        P_x(x) &\text{ if } x \in E\\
    \sum_{Q\in\Lambda_E^\infty}\theta_Q(x)\T_{x_Q}[P_{x_Q}](x) &\text{ if } x \in \Rn\setminus E.
        \end{cases}
        \label{eq:Whitney-extension-def-infinite}
\end{equation}

\begin{theorem}\label{thm:whitney-closed}
    Let $E \subset \Rn$ be a closed set. The map $\T_E : \JF(E) \to \Fs(\Rn)$ defined in \eqref{eq:Whitney-extension-def-infinite} satisfies the following properties.
    \begin{enumerate}[(A)]
        \item $\jet_E\circ\T_E[\Vec{P}]\equiv \Vec{P} $ for all $\Vec{P}\in \JF(E)$, 
        \item $\norm{\T_E[\Vec{P}]}_{\Cs(\Rn)}\leq C(n,s)\norm{\Vec{P}}_{\JC(E)}$,  and
        \item $\norm{\T_E [\Vec{P}]}_{\Fsh(\Rn)} \leq C(n,s)\norm{\Vec{P}}_{\JF(E)}$, and thus, $\norm{\T_E [\Vec{P}]}_{\Fs(\Rn)} \leq C(n,s)\norm{\Vec{P}}_{\JF(E)}$.
    \end{enumerate}
\end{theorem}

The proof of Theorem \ref{thm:whitney-closed} uses similar quantitative estimates as in Theorem \ref{thm:whitney-finite} along with a convergence argument. The details can be found in \cite{BruAYbook,St79,W34-1}.

\section{Finiteness Principles}
\label{sect:fp}

\subsection{Shape Fields}

\begin{definition}\label{def:sf}
    Let $E\subset \Rn$ be a finite set. For each $x \in E$ and $M \geq 0$, let $\Gamma_0(x,M)\subset \P$ be a (possibly empty) convex set. We say $\brac{\Gamma_0(x,M)}_{x\in E, M \geq 0}$ is a \underline{shape field} if 
    \begin{equation*}
        \Gamma_0(x,M)\subset \Gamma_0(x,M')
        \text{ for all } x \in E \text{ and } 0 \leq M \leq M' < \infty.
    \end{equation*}
    Let $A_0, \delta_0 > 0$. We say $\brac{\Gamma_0(x,M)}_{x\in E, M \geq 0}$ is \underline{$(A_0,\delta_0)$-Whitney convex} if given 
    \begin{itemize}
        \item $\delta \in (0,\delta_0]$, $x_0 \in E$, $M \geq 0$,
        \item $P_1, P_2 \in \Gamma_0(x_0,M)$ with $\abs{\grad^m(P_1-P_2)(x_0)} \leq M\delta^{s-m}$ for all $0 \leq m < s$,
        \item $Q_1, Q_2 \in \P$ with 
        $\jet_{x_0}(Q_1^2 + Q_2^2)\equiv 1 $ and
        $\abs{\grad^mQ_j(x_0)} \leq \delta^{-m}$ for $j = 1,2$, $0 \leq m < s$,
    \end{itemize}
    we have
    \begin{equation*}
        P:\equiv \jet_{x_0}\brac{Q_1^2P_1 + Q_2^2P_2} \in \Gamma_0(x_0,A_0M).
    \end{equation*}
     
\end{definition}

By adapting the proof of the Finiteness Principle for shape fields from \cite{FIL16,FIL16+} (see also \cite{F05-J,F05-Sh}) to $\Cs$, we obtain the following result.

\begin{theorem}[$\Cs(\Rn)$ Finiteness Principle for Shape Fields]
\label{thm:sf}
    There exists a large constant $k = k(n,s)$ such that the following hold.

    Let $E\subset \Rn$ be a finite set. Let $(\Gamma_0(x,M))_{x\in E, M \geq 0}$ be a $(A_0,\delta_0)$-Whitney convex shape field. Let $Q_0\subset \Rn$ be a cube with $\delta_{Q_0}\leq \delta_0$. Let $x_0 \in E\cap 5Q_0$ and $M_0 \geq 0$ be given. Suppose for every $S\subset E$ with $\#S \leq k$, there exists a Whitney field $\Vec{P} = (P_x)_{x \in S}$ such that $P_x\in \Gamma_0(x,M)$ for all $x \in S$, and $\norm{\Vec{P}}_{\JCh(S)} \leq M_0$. Then there exist $P_0 \in \Gamma_0(x_0,M_0)$ and $F_0 \in \Cs(Q_0)$ such that
    \begin{enumerate}[(A)]
        \item $\jet_{x}F_0 \in \Gamma_0(x,C(n,s,A_0)M_0)$ for all $E \cap Q_0$,
        \item $\abs{\grad^m(F_0-P_0)(x)} \leq \delta_{Q_0}^{s-m}\cdot C(n,s,A_0)M_0$ for all $x \in Q_0$ and $0 \leq m< s$, and in particular,
        \item $\abs{\grad^m F_0(x)} \leq C(n,s,A_0)M$ for all $x \in Q_0$ and $m = \floor{s}$.
    \end{enumerate}
\end{theorem}

Let $ E\subset \Rn $ be a finite set and let $ f : E \to \pos $. For each $ x \in E $ and $ M \geq 0 $, we define
\begin{equation}
    \Gamma_f(x,M) = \set{P \in 
    \Gamma(x,M): P(x) = f(x)}.
    \label{eq:Gamma-def-f}
\end{equation}

\begin{lemma}\label{lem:Cd convex}
    Let $E\subset \Rn$ be a finite set and let $f : E \to \pos$. Then
    $(\Gamma_f(x,M))_{x\in E, M \geq 0}$ is a $(C,1)$-Whitney convex shape field for some $C = C(n,s)$.
\end{lemma}

\begin{proof}
    It is clear that $\Gamma_f\brac{x,M}$ is convex and $\Gamma_f(x,M)\subset \Gamma_f(x,M')$ whenever $M' \geq M$. It remains to establish $(C,1)$-Whitney convexity.

    Let $0 < \delta \leq 1$, and let $x_0 \in E$. Suppose we are given $P_1, P_2 \in \Gamma_f(x_0,M)$ for some $x_0 \in E$ and $M \geq 0$ satisfying
    \begin{equation}
        \abs{\grad^m (P_1- P_1)(x_0)} \leq M\delta^{s-m}
        \text{ for all } 0 \leq m < s.
        \label{eq:sf-P1}
    \end{equation}
    Suppose we are also given $Q_1, Q_2 \in \P$ such that 
    \begin{equation}
        \abs{\grad^m Q_j (x)} \leq \delta^{-m}
        \text{ for all } 0 \leq m < s \text{ and } j= 1,2,
        \label{eq:sf-Q1}
    \end{equation}
    and moreover,
    \begin{equation}
        \jet_{x_0}\brac{Q_1^2 + Q_2^2} \equiv 1.
        \label{eq:sf-Q2}
    \end{equation}
    We set
    \begin{equation}
        P := \jet_{x_0}\brac{Q_1^2 P_1 + Q_2^2 P_2}
        \label{eq:sf-PQ}
    \end{equation}
    and we would like to show that $P \in \Gamma_f(x_0,C(n,s)M)$. In particular, we need to show that
    \begin{align}
        P(x_0) &= f(x_0), \label{eq:lem:Cd-convex-main-0}\\
        \max_{0 \leq m < s}\abs{\grad^m P(x_0)} &\leq C(n,s)M, \text{ and }
        \label{eq:lem:Cd-convex-main-1}\\
        \max_{0 \leq m < s}\brac{\frac{\abs{\grad^m P(x_0)}}{P(x_0)^{s-m}}}^{1/m} &\leq C(n,s)M.
        \label{eq:lem:Cd-convex-main-2}
    \end{align}

    Note that \eqref{eq:lem:Cd-convex-main-1} follows immediately from \eqref{eq:sf-Q2}, \eqref{eq:sf-PQ}, and the assumption that $P_1, P_2 \in \Gamma_f(x_0,M)$.

    Let $\T_{x_0}$ be as in Lemma \ref{lem:Whitney-single jet}, and we set $F_j:= \T_{x_0}[P_j]$ for $j = 1,2$. Then
    \begin{equation}
        F_j \in \Fs(\Rn), \,
        \norm{F_j}_{\Fs(\Rn)} \leq C(n,s)M, \text{ and }
        \jet_{x_0}F_j \equiv P_j
        \text{ for } j = 1,2.
    \end{equation}

    Thanks to \eqref{eq:sf-Q2}, we may assume (interchanging $Q_1$ and $Q_2$ if necessary) 
    \begin{equation}
        Q_1(x_0) \geq \frac{1}{\sqrt{2}}.
        \label{eq:sf-Q3}
    \end{equation}
    For sufficiently small $c_0 = c_0(n,s)$ (using a similar argument as in Lemma \ref{lem:locally constant}(A)), \eqref{eq:sf-Q1} and \eqref{eq:sf-Q3} yields
    \begin{equation}
        Q_1(x) \geq \frac{1}{10} \text{ for } x \in Q(x_0,c_0\delta),
        \label{eq:sf-Q4}
    \end{equation}
    a cube with center $x_0$ and $\delta_Q = c_0\delta$.
    Fix $c_0$ as in \eqref{eq:sf-Q4}. Let $\chi(x) := \phi_0(\frac{x-x_0}{c_0\delta})$ with $\phi_0 $ as in \eqref{eq:phi_0-def}. Then
    \begin{equation}
        0 \leq \chi \leq 1, \, \supp{\chi}\subset Q(x_0,c_0\delta),\, \jet_{x_0}\chi \equiv 1,\, \text{ and }
        \abs{\grad^m\chi} \leq C(n,x)\delta^{-m}
        \text{ for } 0 \leq m < s.
    \end{equation}
    Define $\phi_1:= \chi\cdot Q_1 + (1-\chi)$ and $\phi_2 = \chi \cdot Q_2$, and $\theta_j:= \frac{\phi_j}{\sqrt{\phi_1^2 + \phi_2^2}}$. Then $\theta_1$ and $\theta_2$ enjoy the following properties.
    \begin{enumerate}[($\theta$-1)]
        \item $\theta_j \in \Cs(\Rn)$ and $\abs{\grad^m\theta_j} \leq C(n,s)\delta^{-m}$ for $j = 1,2$ and $0 \leq m < s$;
        \item $\theta_1^2 + \theta_2^2 = 1$ on $\Rn$, and $\supp{\theta_2} \subset Q(x_0,c_0\delta)$; 
        \item $\jet_{x_0}\theta_j = Q_j$ for $j = 1,2$.
    \end{enumerate}
    We define a nonnegative function $F \in \Cs(\Rn)$ by
    \begin{equation}
        F:= \theta_1^2 F_1 + \theta_2^2F_2 = F_1 + \theta_2^2(F_2-F_1).
        \label{eq:sf-F-def}
    \end{equation}
    We want to show that $F\in \Fs(\Rn)$ with $\norm{F}_{\Fs(\Rn)} \leq C(n,s)M$, $F(x_0) = f(x_0)$, and $\jet_{x_0}F \equiv P$. Note that the latter two are clear from construction, so it suffices to estimate $\norm{F}_{\Fs(\Rn)}$.

    To simplify notation, we write
    \begin{equation*}
        g := \theta_2^2.
    \end{equation*}

    Thanks to \eqref{eq:sf-PQ}, we have $\jet_{x_0}(F_2 - F_1) = P_2 - P_1$. In view of \eqref{eq:sf-P1}, 
    \begin{equation}
        \abs{\grad^m (F_2- F_1)(x_0)}\leq \delta^{s-m}\cdot C(n,s)M
        \text{ for every } 0 \leq m < s.
        \label{eq:sf-Cm-F2-F1}
    \end{equation}
    In view of ($\theta$-1), we see that
    \begin{equation}
\abs{\grad^m\brac{g(F_2 - F_1)}(x_0)} \leq \delta^{s-m}\cdot C(n,s)M
        \text{ for every } 
        0 \leq m < s.
        \label{eq:sf-Cm-est1}
    \end{equation}
    The estimate \eqref{eq:lem:Cd-convex-main-1} follows from \eqref{eq:sf-F-def}, \eqref{eq:sf-Cm-est1}, and the assumption that $\delta < 1$.

Next, we show \eqref{eq:lem:Cd-convex-main-2}.

Suppose $P(x_0) = F(x_0) \geq \delta^s M$. Differentiating \eqref{eq:sf-F-def}, we see that, for $0 \leq m < s$,
\begin{equation}
    \grad^mP(x_0) = \grad^mF(x_0) = \grad^m F_1(x_0) + \sum_{k+l = m,\, k \geq 1}C(k,l)\grad^kg(x_0)\otimes \grad^l(F_2-F_1)(x_0).
    \label{eq:sf-derivatives}
\end{equation}

In view of Lemma \ref{lem:Whitney-single jet}, ($\theta$-1), \eqref{eq:sf-Cm-F2-F1}, and \eqref{eq:sf-derivatives}, we see that
\begin{equation*}
   \begin{split}
        \brac{\frac{\abs{\grad^m P(x_0)}^s}{P(x_0)^{s-m}}}^{1/m}
        &\leq \norm{F_1}_{\Fs(\Rn)} + \brac{\delta^{-ks}\delta^{(s-l)s}\delta^{-(s-m)s}}^{1/m}\cdot C(n,s)M \leq C(n,s)M.
   \end{split}
\end{equation*}

Suppose $F(x_0) < \delta^sM$. Then $F_j(x_0) < \delta^sM$ for $j= 1,2$. Thanks to Lemma \ref{lem:Whitney-single jet} and ($\theta$-1), we have
\begin{equation}
    \begin{split}
    \brac{\frac{\abs{\grad^kg(x_0)\otimes \grad^lF_j(x_0)}^s}{g(x_0)^{(s-m)}F_j(x_0)^{s-m}}}^{1/m}
        &\leq
        C(n,s)\brac{
        g(x_0)^lF_j(x_0)^k
        \frac{\abs{\grad^kg(x_0)}^s}{g(x_0)^{s-m}}
        \frac{\abs{\grad^lF_j(x_0)}^s}{F_j(x_0)^{s-l}}
        }^{1/m}\\
        &\leq C(n,s)g(x_0)^{l/m}F_j(x_0)^{k/m}\norm{g}_{\Fsh(\Rn)}^{k/m}
        \norm{F_j}_{\Fsh(\Rn)}^{l/m}\\
        &\leq \delta^{k/m}\delta^{-k/m}\cdot C(n,s)M = C(n,s)M. 
    \end{split}
    \label{eq:sf-Fsmall-1}
\end{equation}
In view of the identity \eqref{eq:sf-F-def} and \eqref{eq:sf-Fsmall-1}, we see that
\eqref{eq:lem:Cd-convex-main-2} holds. 

\end{proof}

\subsection{Proof of Theorem \ref{thm:fp-arbitrary}}

First, we prove the finite-set version of Theorem \ref{thm:fp-arbitrary}. Recall that, given $S \subset \Rn$ and $f : S \to \pos$, 
\begin{equation*}
    \norm{f}_{\Fs(S)}:= \inf\set{\norm{F}_{\Fs(\Rn)}: F \in \Fs(\Rn) \text{ and } F = f \text{ on }S}.
\end{equation*}

\begin{theorem}[Finiteness Principle for $\Fs$, finite set version]\label{thm:fp-finite}
    There exist constants $k^\sharp = k^\sharp(n,s)$ and $C^\sharp = C^\sharp(n,s)$ such that the following holds. Given any finite subset $E\subset \Rn$ and $f : E \to \pos$,
    \begin{equation}
        \max_{S\subset E,\, \#S \leq k^\sharp }\norm{f}_{\Fs(S)}\leq \norm{f}_{\Fs(E)} \leq C^\sharp \cdot\max_{S\subset E,\, \#S \leq k^\sharp }\norm{f}_{\Fs(S)}.
    \end{equation}
\end{theorem}

\begin{proof}
    The first inequality is immediate. We shall prove the second. We shall prove the case when $E\subset 0.5Q_0$ where $\delta_{Q_0} = 1$. The general case follows immediately if we apply a partition of unity adapted to a tiling of $\Rn$ by unit cubes. 
    Set
    \begin{equation*}
        M_0 := 2\max_{S\subset E,\, \#S \leq k^\sharp }\norm{f}_{\Fs(S)} < \infty.
    \end{equation*}
    It suffices to exhibit a function $F \in \Fs(\Rn)$ with $F = f$ on $E$ and $\norm{F}_{\Fs(\Rn)} \leq C(n,s)M_0$.

    Recall from \eqref{eq:Gamma-def-f} that 
    \begin{equation*}
        \Gamma_f(x,M)= \set{P \in \Gamma(x,M) : P(x) = f(x)}
        \text{ for each } x \in E.
    \end{equation*}
    
    By the definition of $\norm{f}_{\Fs(S)}$, given any $S \subset E$ with $\#S \leq k^\sharp$, there exists $F_S \in \Fs(\Rn)$ such that $F_S = f$ on $S$ and $\norm{F_S}_{\Fs(\Rn)} \leq M_0$. Let $\Vec{P}_S := \jet_SF_S = (\jet_xF_S)_{x \in S}$. 
    By Taylor's theorem \ref{thm:Taylor},
    \begin{equation*}
        \jet_xF_S \in \Gamma_f(x,M_0)
        \text{ for each } x\in S
        \text{ and }
        \norm{\Vec{P}_S}_{\JF(S)} \leq C(n,s)M_0.
    \end{equation*}
    Therefore, our present hypothesis supplies the Whitney field $\vec{P}_S$ required in the hypothesis of Theorem \ref{thm:sf}. Hence, rescaling Lemma \ref{lem:Cd convex} and applying Theorem \ref{thm:sf}, we obtain $P_0 \in \Gamma_f(x_0,C(n,s)M_0)$ and $F_0 \in \Cs(Q_0)$ such that 
    \begin{align}
        &\jet_xF_0 \in \Gamma_f(x,C(n,s)M_0)
        \text{ for all } x \in E, \text{ and }
        \label{eq:fp-F in Gamma}
        \\
        &\norm{\grad^m(F-P_0)}_{L^\infty(Q_0)} \leq C(n,s)M
        \text{ for } 0 \leq m < s.
    \end{align}
    In view of the definitions of $\Gamma$ and $\Gamma_f$ in \eqref{eq:Gamma-def} and \eqref{eq:Gamma-def-f}, we see that
    \begin{equation}
        \norm{F_0}_{\Cs(Q_0)} \leq C(n,s)M_0
        \text{ for } 0 \leq m < s.
        \label{eq:fp-F Cs estimate}
    \end{equation}

    Note that $F_0$ need not be in $\Fs$, but we will only use the jets of $F_0$. 

    Let $\T_E$ be as in Theorem \ref{thm:whitney-finite}. We define
    \begin{equation}
        F:= \T_E\circ \jet_E F_0.
    \end{equation}
    Thanks to Taylor's theorem, Theorem \ref{thm:whitney-finite}, \eqref{eq:fp-F in Gamma}, and \eqref{eq:fp-F Cs estimate} we see that $F \in \Fs(\Rn)$, $\jet_x F = \jet_xF_0 \in \Gamma_f(x,C(n,s)M)$ for every $x \in E$, and $\norm{F}_{\Fs(\Rn)} \leq C(n,s)M_0$. In particular, $F(x) = f(x)$ for every $x \in E$. Therefore, $\norm{f}_{\Fs(E)} \leq C(n,s)M_0$.
\end{proof}

Next, we use a compactness argument to pass from finite $E\subset \Rn$ to arbitrary $E\subset \Rn$.

\begin{proof}[Proof of Theorem \ref{thm:fp-arbitrary}]

    Without loss of generality, we may assume that $E\subset \frac{1}{2}Q_0$ with $\delta_{Q_0} = 1$. It suffices to show the second inequality in \eqref{eq:fp-closed-quant}. Set $M_0 := 2 \sup\limits_{S\subset E,\, \#S \leq k^\sharp }\norm{f}_{\Fs(S)}$. It suffices to exhibit a function $F \in \Fs(\Rn)$ with $F = f$ on $E$ and $\norm{F}_{\Fs(\Rn)} \leq C(n,s)M_0$.

    By Ascoli's theorem,
    \begin{equation*}
        \mathcal{B}:= \set{F \in \Fs(Q_0) : \norm{F}_{\Fs(Q_0)} \leq C(n,s)M_0}
    \end{equation*}
    is compact in the non-H\"older $C^{\floor{s}}(Q_0)$ norm topology 
    \footnote{given by $\norm{F}_{C^{\floor{s}}(Q_0)}:= \max\limits_{0 \leq m < s}\sup\limits_{x\in Q_0}\abs{\grad^m F(x)}$}. For each finite $E_0 \subset E$, Theorem \ref{thm:fp-finite} tells us that there exists $F_{E_0} \in \mathcal{B}$ such that $F_{E_0} = f$ on $E_0$. Consequently, there exists $F \in \mathcal{B}$ such that $F = f$ on $E$. 
    
\end{proof}

\appendix

\section{Proof of Theorem \ref{thm:root}}
 \label{section:thm:root}



 Given a multi-index $\alpha = (\alpha_1, \cdots, \alpha_n) \in \mathbb{N}_0^n$, let $\Pi(\alpha)$ denotes all the possible partition of $\alpha$, i.e.,
 \begin{equation*}
     \Pi(\alpha) = \set{\pi = \set{\beta_1, \cdots, \beta_k} : k \in \mathbb{N}, \beta_j \in \mathbb{N}_0^n \text{ for all }j, \text{ and }\sum_{j=1}^k\beta_j = \alpha}.
 \end{equation*}

 \begin{theorem}[Multivariate Fa\'a di Bruno's Formula, \cite{Bruno-Hardy}]
 \label{thm:FdB}
Let $I\subset \Rn$ be an interval. Let $h : I \to \R$ and $F: \Rn \to I $ be sufficiently smooth. Then
\begin{equation}
    \da(h\circ F) = \sum_{\pi \in \Pi(\alpha)} \brac{h^{(\abs{\pi})} \circ F} \cdot \prod_{\beta \in \pi}\d^\beta F.
    \label{eq:FdB}
\end{equation}
Here, $\abs{\pi}$ is the cardinality of $\pi$ and $h^{(k)}$ is the $k$-th derivative of $h$.
\end{theorem}


We are mainly interested in Theorem \ref{thm:FdB} for case $h(t) = t^r$ for some $r \in (0,1)$. Thus, $h^{(k)} = C(k,r)t^{r-k}$ for some $C(k,r) \neq  0$.

\begin{lemma}\label{lem:key-flat}
    Given $r \in (0,1)$, $s > 0$, $\Omega \subset \Rn$ convex, $F \in \Fs(\Omega)$, and $x,y \in \Omega$, we have
    \begin{equation}
        \abs{\grad^m F^r(x)} \leq C(n,r,s)\norm{F}_{\Fs(\Omega)}^{m/s}F(x)^{r-m/s}
        \text{ for all $0 \leq m < s$,}
        \label{eq:lem:key-flat-1}
    \end{equation}
    and
    \begin{equation}
        \abs{\grad^{\floor{s}}F^r(x) - \grad^{\floor{s}}F^r(y)} \leq C(n,r,s) \frac{
        \brac{\norm{F}_{\Csh(\Omega)} + \norm{F}_{\Fsh(\Omega)}}
        }{
        \min\set{F(x)^{1-r}, F(y)^{1-r}}
        }
        \abs{x-y}^{s-\floor{s}}.
        \label{eq:lem:key-flat-2}
    \end{equation}
    Moreover, if $F \geq \epsilon > 0$, then
    \begin{equation}
        \norm{F^r}_{\Fs(\Omega)} \leq C(n,r,s)\epsilon^{r-1}\norm{F}_{\Fs(\Omega)}.
        \label{eq:lem:key-flat-3}
    \end{equation}
\end{lemma}

\begin{proof}
\newcommand{\db}{\d^\beta}
    We adapt the proof of Lemma 2 in \cite{RSH17} to the combinatorial form of the multivariate Fa\'a di Bruno's formula.

    If $0 < s \leq 1$, then \eqref{eq:lem:key-flat-1} is trivial and \eqref{eq:lem:key-flat-2} follows from the mean value theorem. From now on, we assume $s > 1$, and we set 
    \begin{equation*}
        m_0:= \floor{s}.
    \end{equation*}

    Without loss of generality, we may assume $F(y)\leq F(x)$ and $\norm{F}_\Csh + \norm{F}_\Fsh = 1$.

    First of all, by the definition of $\norm{F}_\Fsh$,
    \begin{equation}
        \abs{\prod_{\beta \in \pi}\db F(x)} \leq 
        \prod_{\beta \in \pi} F(x)^{\frac{s-\abs{\beta}}{s}} = F(x)^{\abs{\pi}-\frac{m_0}{s}}.
        \label{eq:lem:key-flat-prod}
    \end{equation}

    Applying Fa\'a di Bruno's Formula \eqref{eq:FdB} and \eqref{eq:lem:key-flat-prod}, we see that
    \begin{equation}
        \abs{\da F^r(x)} \leq C(n,r,s) \sum_{\pi \in \Pi(\alpha)} F(x)^{r-\abs{\pi}} F(x)^{\abs{\pi} - \frac{m_0}{s}} \leq C(n,r,s)F(x)^{r-\frac{m_0}{s}}.
        \label{eq:case2-1}
    \end{equation}

     We see that \eqref{eq:lem:key-flat-1} follows from \eqref{eq:case2-1}.

     For the rest of the proof, we fix an arbitrary multi-index $\alpha$ with $\abs{\alpha} = m_0$ and a partition $\pi$ of $\alpha$. Note that the following arguments hold for $\Tilde{\alpha}$ with $\abs{\Tilde{\alpha}} < m_0 $ after suitable simplification. 

    \paragraph{Case I: suppose $\abs{x-y}\leq c_0F(x)^{1/s}$} with $c_0$ as in Lemma \ref{lem:locally constant} with $\epsilon = 1/2$.

    Thus, $F(y)/2 \leq F(x) \leq 3F(y)/2$, which we will use freely throughout rest of the proof. 

    By the triangle inequality,
    \begin{equation}
        \begin{split}
            &\abs{F(x)^{r-\abs{\pi}}
            \prod_{\beta \in \pi}
            \db F(x)
            -
            F(y)^{r-\abs{\pi}}
            \prod_{\beta \in \pi}
            \db F(y)
            }
            \\
            &\quad\quad
            \leq
            \abs{
            \brac{F(x)^{r-\abs{\pi}} - F(y)^{r-\abs{\pi}}}
            \prod_{\beta\in\pi}\db F(x)
            }
            + 
            \abs{
            F(y)^{r-\abs{\pi}}
            \brac{
            \prod_{\beta \in \pi}\db F(x) - \prod_{\beta\in\pi} \db F(y)
            }
            }\\
            &=: M_1 + M_2.
            \label{eq:M1+M2}
        \end{split}
    \end{equation}

    Let $u = \frac{y-x}{\abs{y-x}}$. By the mean value theorem applied to $F$ along the $u$-direction, there exists a point $\xi = \xi_{x,y,\alpha}$ on the open segment connecting $x$ and $y$ such that
    \begin{equation}
        \begin{split}
            &\abs{F(x)^{r-\abs{\pi}} - F(y)^{r-\abs{\pi}}}
            = \abs{F(x)^{r-\abs{\pi}} - F(y)^{r-\abs{\pi}}}^{(s-m_0) + (1-(s-m_0))}\\
            &\leq \abs{(r-\abs{\pi})F(\xi_{})^{r-\abs{\pi}-1}
            \grad_u F(\xi_{})\abs{x-y}
            }^{s-m_0}
            \cdot
            \abs{F(x)^{r-\abs{\pi}} + F(y)^{r-\abs{\pi}}}^{1-(s-m_0)}
            \\
            &\leq C(n,r,s)\abs{F(y)}^{-\abs{\pi}-1 + r + m_0/s}\abs{x-y}^{s-m_0}.
        \end{split}
        \label{eq:lem:key-flat-root}
    \end{equation}
    In light of \eqref{eq:lem:key-flat-prod} and \eqref{eq:lem:key-flat-root}, we see that
    \begin{equation}
        M_1 \leq C(n,r,s)F(y)^{r-1}\abs{x-y}^{s-m_0}.
        \label{eq:M1-final}
    \end{equation}

    Using a similar argument as in \eqref{eq:lem:key-flat-root}, we see that for $\abs{\beta}\leq m_0-1$
    \begin{equation}
        \begin{split}
            \abs{\db F(x) - \db F(y)} 
            &\leq 
            \abs{\grad_u \db F(\xi)\abs{x-y}}^{s-m_0}\abs{\db F(x) - \db F(y)}^{1-(s-m_0)}\\
            &\leq C(n,r,s)\abs{F(\xi)}^{\frac{s-\abs{\beta}-1}{s}(s-m_0)}\abs{F(y)}^{\frac{s-\abs{\beta}}{s}(1-s+k)}\abs{x-y}^{s-m_0}\\
            &\leq C(n,r,s)\abs{F(y)}^{\frac{m_0-\abs{\beta}}{s}}\abs{x-y}^{s-m_0}. 
        \end{split}
        \label{eq:lem:key-flat-root-der}
    \end{equation}
    For $\abs{\beta} = m_0$, since $F \in \Cs$, we have
    \begin{equation}
        \abs{\db F(x) - \db F(y)} \leq \norm{F}_{\Csh}\abs{x-y}^{s-m_0} \leq \abs{x-y}^{s-m_0} = \abs{F(y)}^{\frac{m_0-\abs{\beta}}{s}}\abs{x-y}^{s-m_0}.
        \label{eq:lem:key-flat-root-der-top}
    \end{equation}

    Now we estimate $M_2$.

    Temporarily fix $\beta \in \pi$, $x,y\in \Omega$. We define two arrays $(z_{\gamma,\beta})_{\gamma\in\pi}$ and $(w_{\gamma,\beta})_{\gamma\in\pi}$ with the following properties:
    \begin{itemize}
        \item $z_{\beta,\beta} = x$ and $w_{\beta,\beta} = y$;
        \item Either $z_{\beta,\gamma} = w_{\beta,\gamma} = x$ for all $\gamma \neq \beta$ or $z_{\beta,\gamma} = w_{\beta,\gamma} = y$ for all $\gamma \neq \beta$. 
    \end{itemize}
    Using \eqref{eq:lem:key-flat-root-der} and \eqref{eq:lem:key-flat-root-der-top}, we have
    \begin{equation}
        \begin{split}
            &F(y)^{r-\abs{\pi}}\abs{
            \prod_{\gamma\in\pi}\d^\gamma F(z_{\beta,\gamma}) - \prod_{\gamma \in \pi}\d^{\gamma}F(w_{\beta,\gamma})
            }\\
            &= 
            F(y)^{r-\abs{\pi}}\prod_{\gamma \in \pi, \gamma \neq \beta} 
            \abs{\d^\gamma F(z_{\beta,\gamma})}
            \abs{
            \db F(x) - \db F(y)
            }   \\
            &\leq 
            C(n,r,s)
            F(y)^{r-\abs{\pi}}
            \brac{ \prod_{\gamma \in \pi, \gamma \neq \beta} 
            F(z_{\beta,\gamma})^{\frac{s-\abs{\gamma}}{s}}
            }
            {F(y)}^{\frac{(s-\abs{\beta})-(s-m_0)}{s}}
            \abs{x-y}^{s-m_0}\\
            &\leq 
            C(n,r,s)
            F(y)^{r-1}\abs{x-y}^{s-m_0}.
        \end{split}
        \label{eq:telescope}
    \end{equation}
    In the last step, we used the fact that $\sum_{\gamma \in \pi}\gamma = \abs{\alpha} = m_0$. 
    By repeatedly applying triangle inequality and using \eqref{eq:telescope}, we have
    \begin{equation}
        \begin{split}
            M_2 &\leq F(y)^{r-\abs{\pi}}\sum_{\beta \in \pi}\abs{
            \prod_{\gamma \in \pi}\d^\gamma F(z_{\beta,\gamma})
            - 
            \prod_{\gamma \in \pi}\d^\beta F(w_{\beta,\gamma})
            } \leq C(n,r,s) F(y)^{r-1}\abs{x-y}^{s-m_0}.
        \end{split}
        \label{eq:M2-final}
    \end{equation}
    In light of Fa\'a di Bruno's Formula \eqref{eq:FdB}, estimates \eqref{eq:M1+M2}, \eqref{eq:M1-final}, and \eqref{eq:M2-final}, we see that \eqref{eq:lem:key-flat-2} holds in the case $\abs{x-y} \leq c(n,s)F(x)^{1/s}$.

    \paragraph{Case II: suppose $\abs{x-y} > c_0F(x)^{1/s}$} with $c_0$ as in Lemma \ref{lem:locally constant} with $\epsilon = 1/2$. Equivalently, $F(x) \leq C(n,s)\abs{x-y}^s$.

    Since $F(y) \leq F(x) \leq C(n,s)\abs{x-y}^s$, \eqref{eq:case2-1} implies
    \begin{equation}
        \begin{split}
            \abs{\da F(x) - \da F(y)}
        &\leq 
        C(n,r,s) \brac{ F(x)^{r-m_0/s} + F(y)^{r-m_0/s} }\\
        &\leq C(n,r,s)\frac{1}{F(y)^{1-r}}F(x)^{(s-m_0)/s}\\
        &\leq C(n,r,s) \frac{1}{F(y)^{1-r}}\abs{x-y}^{s-m_0}.
        \end{split}
        \label{eq:case2-2}
    \end{equation}
    We see that \eqref{eq:lem:key-flat-2} follows from \eqref{eq:case2-2} in the case $\abs{x-y}> c(n,s)F(x)^{1/s}$.

    Thus, we have shown \eqref{eq:lem:key-flat-2} for arbitrary $x,y \in \Omega$. 

    Finally, we turn to \eqref{eq:lem:key-flat-3}. Suppose $F \geq \epsilon > 0$ on $\Omega$. It follows from \eqref{eq:lem:key-flat-2} and \eqref{eq:lem:key-flat-1} that 
    \begin{equation}
        \norm{F^r}_{\Cs} \leq C(n,r,s)\epsilon^{r-1}\norm{F}_{\Fs}
        \label{eq:last-1}
    \end{equation}
    On the other hand, we see from \eqref{eq:lem:key-flat-1} that
    \begin{equation}
        \begin{split}
            \abs{\grad^mF^r(x)}
            &\leq C(n,r,s)\norm{F}_{\Fs}^{m/s}F(x)^{r-m/s}\brac{F(x)/\epsilon}^{\frac{m(1-r)}{s}}\\
            &= C(n,r,s)\brac{\norm{F}_{\Fs}/\epsilon^{1-r} }^{m/s} \brac{ F(x)^r }^{\frac{s-m}{s}}.
        \end{split}
        \label{eq:last-2}
    \end{equation}
    In light of \eqref{eq:last-1} and \eqref{eq:last-2}, we see that \eqref{eq:lem:key-flat-3} holds. 
\end{proof}

\begin{proof}[Proof of Theorem \ref{thm:root}]

\newcommand{\Csr}{\mathcal{C}^{sr}}
\newcommand{\Csrh}{\dot{\mathcal{C}}^{sr}}
\newcommand{\Fshr}{\dot{\mathcal{F}}^{sr}}

By rescaling, we may assume that $\norm{F}_{\Fs} = 1$.

Thanks to \eqref{eq:lem:key-flat-1}, we see that
\begin{equation}
    \max_{0 \leq m < \floor{sr}}\sup_{x \in \Omega}\abs{\grad^m F^r(x)} \leq C(n,r,s)
    \text{ and }
    \norm{F^r}_{\Fshr} \leq C(n,r,s).
\end{equation}
It remains to show that 
\begin{equation}
    \norm{F^r}_{\Csrh(\Omega)} \leq C(n,r,s).
    \label{eq:thm:root-main}
\end{equation}

We write
\begin{equation*}
    sr = m_0 + \sigma 
    \text{ with } m_0 = \floor{sr} \in \mathbb{Z} \text{ and }\sigma \in (0,1].
\end{equation*}

Let $x,y \in \Omega$. Let $\ell_{x,y}$ be the line segment connecting $x$ and $y$. Set
\begin{equation*}
    \Delta_{r,m_0}(x,y) := \abs{\grad^{m_0}F^r(x) - \grad^{m_0}F^r(y)}.
\end{equation*}

Thanks to \eqref{eq:lem:key-flat-1}, we have
\begin{equation}
    \abs{\grad^{m_0}F^r(x)} \leq C(n,r,s)F(x)^{r - m_0/s}.
    \label{eq:thm:root-1}
\end{equation}

Suppose $\abs{x-y} \geq c_0\max\set{F(x), F(y)}^{1/s}$ with $c_0$ as in Lemma \ref{lem:locally constant} with $\epsilon = 1/2$. We use \eqref{eq:thm:root-1} to obtain
\begin{equation}
    \Delta_{r,m_0}(x,y) \leq C(n,r,s)\max\set{F(x), F(y)}^{r-m_0/s} \leq C(n,r,s)\abs{x-y}^{\sigma}.
\end{equation}

Suppose $\abs{x-y} < c_0\max\set{F(x), F(y)}^{1/s}$. We see from Lemma \ref{lem:locally constant} that 
\begin{equation}
    F(x)/2 \leq F(\xi) \leq 2F(x)
    \text{ for all } \xi \in B(x,\abs{x-y}).
    \label{eq:thm:root:proof-locally constant}
\end{equation}

Suppose $m_0 < \floor{s}$. 
Thanks to \eqref{eq:lem:key-flat-1}, \eqref{eq:thm:root:proof-locally constant}, and the mean value theorem, 
\begin{equation}
    \begin{split}
       \Delta_{r,m_0}(x,y)
        &\leq \sup_{\xi \in \ell_{x,y}}\abs{\grad^{m_0+1}F^r(\xi)}
        \abs{y-x}
        \leq C(n,r,s)F(x)^{r-(m_0+1)/s}\abs{y-x}.
    \end{split}
    \label{eq:thm:root-2}
\end{equation}
By writing $\Delta_{r,m_0}(x,y) = \Delta_{r,m_0}(x,y)^{1-\sigma}\Delta_{r,m_0}(x,y)^\sigma$, we can use \eqref{eq:thm:root-1} and \eqref{eq:thm:root-2} to bound the first and second factor, respectively, to obtain $\Delta_{r,m_0}(x,y) \leq C(n,r,s)\abs{x-y}^{\sigma}$.

On the other hand, suppose $m_0 = \floor{s}$. Let
\begin{equation*}
    \theta = \frac{\sigma}{s - m_0}.
\end{equation*}
By writing $\Delta_{r,m_0}(x,y) = \Delta_{r,m_0}(x,y)^{1-\theta}\Delta_{r,m_0}(x,y)^\theta$, we can use \eqref{eq:thm:root-1} to bound the first factor and \eqref{eq:lem:key-flat-2} to bound the second to obtain $\Delta_{r,m_0}(x,y) \leq C(n,r,s)\abs{x-y}^{\sigma}$.

Therefore, we have shown \eqref{eq:thm:root-main}. This concludes the proof of Theorem \ref{thm:root}.
    
\end{proof}

 \bibliographystyle{plain}
 \bibliography{Whitney}

\begin{thebibliography}{10}

\bibitem{BM07}
Edward Bierstone and Pierre~D. Milman.
\newblock $\mathscr{C}^m$-norms on finite sets and $\mathscr{C}^m$-extension
  criteria.
\newblock {\em Duke Math. J.}, 137(1):1--18, 2007.

\bibitem{BMP06}
Edward Bierstone, Pierre~D. Milman, and Wiesław Pawłucki.
\newblock Higher-order tangents and fefferman's paper on {W}hitney's extension
  problem.
\newblock {\em Ann. of Math. (2)}, 164(1):361--370, 2006.

\bibitem{root-Bony06}
Jean-Michel Bony, Ferruccio Colombini, and Ludovico Pernazza.
\newblock Nonnegative functions as squares or sums of squares.
\newblock {\em J. Funct. Anal.}, 232(1):137--147, 2006.

\bibitem{root-Bony10}
Jean-Michel Bony, Ferruccio Colombini, and Ludovico Pernazza.
\newblock On square roots of class ${C}^m$ of nonnegative functions of one
  variable.
\newblock {\em Ann. Sc. norm. super. Pisa - Cl. sci. Serie V}, 3, 09 2010.

\bibitem{BruAYbook}
Alexander Brudnyi and Yuri Brudnyi.
\newblock {\em Methods of Geometric Analysis in Extension and Trace Problems},
  volume 102 of {\em Monographs in Mathematics}.
\newblock Birkh\"auser, 2010.

\bibitem{BS94-W}
Yuri Brudnyi and Pavel Shvartsman.
\newblock Generalizations of {W}hitney's extension theorem.
\newblock {\em Internat. Math. Res. Notices}, 3(129), 1994.

\bibitem{BS94-Tr}
Yuri Brudnyi and Pavel Shvartsman.
\newblock The traces of differentiable functions to subsets of
  {$\mathbb{R}^n$}.
\newblock In {\em Linear and complex analysis. Problem book 3}, volume 1574 of
  {\em Lecture Notes in Mathematics}, pages 279--281. Springer-Verlag, Berlin,
  1994.

\bibitem{BS98}
Yuri Brudnyi and Pavel Shvartsman.
\newblock The trace of jet space {$J^k\Lambda^\omega$} to an arbitrary closed
  subset of {$\mathbb{R}^n$}.
\newblock {\em Trans. Amer. Math. Soc.}, 350(4):1519--1553, 1998.

\bibitem{BS01}
Yuri Brudnyi and Pavel Shvartsman.
\newblock {W}hitney's extension problem for multivariate
  {${C}^{1,\omega}$}-functions.
\newblock {\em Trans. Amer. Math. Soc.}, 353(6):2487--2512 (electronic), 2001.

\bibitem{F05-J}
Charles Fefferman.
\newblock A generalized sharp {W}hitney theorem for jets.
\newblock {\em Rev. Mat. Iberoam.}, 21(2):577--688, 2005.

\bibitem{F05-L}
Charles Fefferman.
\newblock Interpolation and extrapolation of smooth functions by linear
  operators.
\newblock {\em Rev. Mat. Iberoam.}, 21(1):313--348, 2005.

\bibitem{F05-Sh}
Charles Fefferman.
\newblock A sharp form of {W}hitney's extension theorem.
\newblock {\em Ann. of Math. (2)}, 161(1):509--577, 2005.

\bibitem{F06}
Charles Fefferman.
\newblock {W}hitney's extension problem for {$ {C^m} $}.
\newblock {\em Ann. of Math. (2)}, 164(1):313--359, 2006.

\bibitem{F07-L}
Charles Fefferman.
\newblock {${C}^m$} extension by linear operators.
\newblock {\em Ann. of Math. (2)}, 166(2):779--835, 2007.

\bibitem{F09-Data-3}
Charles Fefferman.
\newblock Fitting a {$C^m$}-smooth function to data {III}.
\newblock {\em Ann. of Math. (2)}, 170(1):427--441, 2009.

\bibitem{F09-Int}
Charles Fefferman.
\newblock {W}hitney's extension problems and interpolation of data.
\newblock {\em Bull. Amer. Math. Soc. (N.S.)}, 46(2):207--220, 2009.

\bibitem{FI20-book}
Charles Fefferman and Arie Israel.
\newblock {\em Fitting Smooth Functions to Data}.
\newblock {CBMS} Regional Conference Series in Mathematics. American
  Mathematical Society, 2020.

\bibitem{FIL16}
Charles Fefferman, Arie Israel, and Garving~K. Luli.
\newblock Finiteness principles for smooth selections.
\newblock {\em Geom. Funct. Anal.}, 26(2):422--477, 2016.

\bibitem{FIL16+}
Charles Fefferman, Arie Israel, and Garving~K. Luli.
\newblock Interpolation of data by smooth non-negative functions.
\newblock {\em Rev. Mat. Iberoam.}, 33(1):305—324, 2016.

\bibitem{FJL23}
Charles Fefferman, Fushuai Jiang, and Garving~K. Luli.
\newblock ${C}^2$ interpolation with range restriction.
\newblock {\em Rev. Mat. Iberoam.}, 39(2):649--710, 2023.

\bibitem{FIL13}
Charles Fefferman, Garving~K. Luli, and Arie Israel.
\newblock {S}obolev extension by linear operators.
\newblock {\em Journal A.M.S.}, 27(1):69--145, 2013.

\bibitem{G58}
Georges Glaeser.
\newblock Étude de quelques algèbres tayloriennes.
\newblock {\em J. Analyse Math.}, 6:1--124, 1958.

\bibitem{G63}
Georges Glaeser.
\newblock Racine carr\'ee d'une fonction diff\'erentiable.
\newblock {\em Annales de l'Institut Fourier}, 13(2):203--210, 1963.

\bibitem{Bruno-Hardy}
Michael Hardy.
\newblock Combinatorics of partial derivatives.
\newblock {\em Electron. J. Comb}, 13, 2006.

\bibitem{I13}
Arie Israel.
\newblock A bounded linear extension operator for $ {L}^{2,p}(\mathbb{R}^2) $.
\newblock {\em Ann. of Math.}, 178(1):183--230, 2013.

\bibitem{JLLL23}
Fushuai Jiang, Chen Liang, Yutong Liang, and Garving~K. Luli.
\newblock Univariate range-restricted ${C}^2$ interpolation algorithms.
\newblock {\em J. Comput. and Appl. Math.}, 425:115040, 2023.

\bibitem{JL20-Ext}
Fushuai Jiang and Garving~K. Luli.
\newblock ${C^2(\mathbb{R}^2)}$ nonnegative interpolation by bounded-depth
  operators.
\newblock {\em Adv. Math.}, 375:107391, 2020.

\bibitem{JL20}
Fushuai Jiang and Garving~K. Luli.
\newblock Nonnegative $ {C}^2(\mathbb{R}^2)$ interpolation.
\newblock {\em Adv. Math.}, 375:107364, 2020.

\bibitem{JL20-Alg}
Fushuai Jiang and Garving~K. Luli.
\newblock Algorithms for nonnegative ${C^2(\mathbb{R}^2)}$ interpolation.
\newblock {\em Adv. Math.}, 385:107756, 2021.

\bibitem{JLO20}
Fushuai Jiang, Garving~K. Luli, and Kevin O'Neill.
\newblock On the shape fields finiteness principle.
\newblock {\em Int. Mat. Res. Not.}, 2022(23):18895--18918, 2021.

\bibitem{JLO22}
Fushuai Jiang, Garving~K. Luli, and Kevin O'Neill.
\newblock Smooth selection for infinite sets.
\newblock {\em Adv. Math.}, 407:108566, 2022.

\bibitem{St79}
Elias M.~Stein.
\newblock {\em Singular Integrals and Differentiability Properties of
  Functions}, volume~2 of {\em Monographs in harmonic analysis}.
\newblock Princeton University Press, Princeton, NJ, 1970.

\bibitem{noise}
Michael Nussbaum.
\newblock {Asymptotic equivalence of density estimation and Gaussian white
  noise}.
\newblock {\em Ann. Stat.}, 24(6):2399 -- 2430, 1996.

\bibitem{RSH16}
Kolyan Ray and Johannes Schmidt-Hieber.
\newblock Minimax theory for a class of nonlinear statistical inverse problems.
\newblock {\em Inverse Probl.}, 32(6):065003, 2016.

\bibitem{RSH17}
Kolyan Ray and Johannes Schmidt-Hieber.
\newblock A regularity class for the roots of nonnegative functions.
\newblock {\em Annali di Matematica}, 196:2091--2103, 2017.

\bibitem{Shv08}
Pavel Shvartsman.
\newblock The {W}hitney extension problem and {L}ipschitz selections of
  set-valued mappings in jet-spaces.
\newblock {\em Trans. Amer. Math. Soc.}, 360(10):5529--5550, 2008.

\bibitem{W34-1}
Hassler {W}hitney.
\newblock Analytic extensions of differentiable functions defined in closed
  sets.
\newblock {\em Trans. Amer. Math. Soc.}, 36(1):63--89, 1934.

\bibitem{W34-2}
Hassler {W}hitney.
\newblock Differentiable functions defined in closed sets. {I}.
\newblock {\em Trans. Amer. Math. Soc.}, 36(2):369--387, 1934.

\bibitem{W34-3}
Hassler {W}hitney.
\newblock Functions differentiable on the boundaries of regions.
\newblock {\em Ann. of Math. (2)}, 35(3):482--485, 1934.

\bibitem{Z98}
Nahum Zobin.
\newblock {W}hitney's problem on extendability of functions and an intrinsic
  metric.
\newblock {\em Adv. Math.}, 133(1):96--132, 1998.

\bibitem{Z99}
Nahum Zobin.
\newblock Extension of smooth functions from finitely connected planar domains.
\newblock {\em J. Geom. Anal.}, 9(3):489--509, 1999.

\end{thebibliography}


\end{document}